\documentclass[a4j, 11pt, draft]{amsart}
\usepackage{amsmath,amssymb}	
\usepackage{amsthm}
\usepackage{mathrsfs}	
\usepackage{bm}	

\allowdisplaybreaks[3]

\theoremstyle{plain} 
\newtheorem{theorem}{Theorem}
\newtheorem{lemma}{Lemma}
\newtheorem{corollary}{Corollary}
\newtheorem{conjecture}{Conjecture}

\newtheorem{proposition}{Proposition}

\newtheorem*{conjecture*}{Conjecture}

\theoremstyle{plain}

\theoremstyle{remark}
\newtheorem{remark}{Remark}

\theoremstyle{definition}
\newtheorem*{acknowledgment*}{Acknowledgments}

\numberwithin{equation}{section}

\DeclareMathOperator*{\psum}{\sideset{}{^{\prime}}\sum}

\DeclareMathOperator*{\Res}{Res}

\DeclareMathOperator{\Gal}{Gal}

\DeclareMathOperator{\rad}{rad}

\newcommand{\QQ}{\mathbb{Q}}
\newcommand{\ZZ}{\mathbb{Z}}

\newcommand{\s}{\sigma}

\newcommand{\mf}{\mathfrak}

\renewcommand{\a}{\alpha}

\renewcommand{\l}{\left}
\renewcommand{\r}{\right}
\renewcommand{\d}{\displaystyle}
\renewcommand{\Re}{\mathrm{Re}}
\renewcommand{\Im}{{\rm Im}}
\renewcommand{\epsilon}{\varepsilon}
\renewcommand{\bar}{\overline}

\newcommand{\todaye}{\the\year/\the\month/\the\day}

\allowdisplaybreaks[0]

\title[Some explicit formulas for partial sums of M\"obius functions]
{Some explicit formulas for partial sums of\\ M\"obius functions}
\author{Sh\={o}ta Inoue}

\address{Graduate School of Mathematics, Nagoya University,
Furocho, Chikusaku, Nagoya 464-8602, Japan}
\email{m16006w@math.nagoya-u.ac.jp}



\begin{document}

\maketitle

\begin{abstract}
The purpose of this paper is to give some explicit formulas involving M\"obius functions, 
which may be known under the generalized Riemann Hypothesis, but unconditional
in this paper.
Concretely, we prove explicit formulas of partial sums of the M\"obius function in arithmetic progressions and partial sums of the M\"obius functions on an Abelian number field $K$. 
In addition, to obtain these explicit formulas, we study a certain finite Euler product appearing from certain relation of primitive characters and imprimitive characters in the present paper.
\end{abstract}



\section{\textbf{Introduction and statement of results}}


The classical explicit formula
\begin{align}	\label{CEFMRH}
M^{*}(x) 
= \lim_{\nu \rightarrow \infty}\sum_{|\gamma| < T_{\nu}}\frac{x^{\rho}}{\zeta'(\rho)\rho} - 2 
+ \sum_{n = 1}^{\infty}\frac{(-1)^{n-1}(2\pi / x)^{2n}}{(2n)!n\zeta(2n + 1)}
\end{align}
was shown by Titchmarsh \cite{T} under the assumption of the Riemann Hypothesis and the simplicity of zeros of the Riemann zeta-function $\zeta(s)$,
where we define $M^{*}(x)$ by
\begin{align*}
M^{*}(x) = \psum_{n \leq x}\mu(n),
\end{align*}
$\mu(n)$ is the M\"obius function, $\sum'$ indicates that if $x$ is an integer, then the last term is to be counted with weight $1/2$, 
and $\{T_{\nu}\}_{\nu = 1}^{\infty}$ is a certain sequence satisfying $T_{\nu} \in [\nu, \nu + 1]$.
In addition, Bartz \cite{B} unconditionally proved the explicit formula
\begin{align}	\label{CEFMU}
M^{*}(x)
=& \lim_{\nu \rightarrow \infty}\sum_{|\gamma| < T_{\nu}}\frac{1}{(m(\rho) - 1)!}\lim_{s \rightarrow \rho}
\frac{d^{m(\rho - 1)}}{ds^{m(\rho - 1)}}\l( (s - \rho)^{m(\rho)}\frac{x^{s}}{\zeta(s)s} \r)\\ \nonumber
&- 2 + \sum_{n = 1}^{\infty}\frac{(-1)^{n-1}(2\pi / x)^{2n}}{(2n)!n\zeta(2n + 1)},
\end{align}
where $\{T_{\nu}\}_{\nu = 1}^{\infty}$ is a certain sequence satisfying $T_{\nu} \in [\nu, 2\nu]$.
It is difficult to apply these explicit formulas because there are some inconvenient points.
For example, the main term of formula \eqref{CEFMU} is more complicated for higher multiplicity zeros, and it is difficult to understand the behavior of multiplicity of nontrivial zeros.
We do not know even the boundedness of multiplicity at present. 
Even if we assume the simplicity of zeros, there is another problem, 
that is the behavior of $\zeta'(\rho)$.
This is also difficult because this problem is related 
to the detailed information on the gaps
between zeros of the Riemann zeta-function.
Here, the following conjecture is known for this problem.

\begin{conjecture*}[The Gonek-Hejhal Conjecture]
Assume the simple zero conjecture for the Riemann zeta-function.
For $\lambda > -\frac{3}{2}$, 
\begin{align*}
\sum_{0 < \gamma < T}\l|\zeta'(\rho)\r|^{2\lambda}
\asymp T(\log{T})^{(\lambda + 1)^2}.
\end{align*}
\end{conjecture*}

This conjecture was independently suggested by Gonek \cite{G} and by Hejhal \cite{H}.
By applying this conjecture and the Riemann Hypothesis to a certain truncated form of
\eqref{CEFMRH}, Ng \cite{ng1} proved the following sharp estimate
\begin{align*}
M(x) \ll x^{1/2}(\log{x})^{5/4}.
\end{align*}
This estimate is stronger than the result 
\begin{align*}
M(x) \ll x^{1/2}\exp\l( (\log{x})^{1/2}(\log{\log{x}})^{14} \r),
\end{align*}
which Soundararajan \cite{SM} showed under only the Riemann Hypothesis.
From the above background, it can be seen that the truncated explicit formulas are important to obtain the exact upper bound for the summatory functions of M\"obius functions.

The present paper gives some truncated explicit formulas, 
which generalize the truncated form of \eqref{CEFMU}.
Our first purpose is to obtain the explicit formula for the function
\begin{align*}
M^{*}(x; q, a) = \underset{n \equiv a \bmod{q}}{\psum_{n \leq x}}\mu(n),
\end{align*}
which is the summatory function of the M\"obius function in arithmetic progressions with $(a, q) = 1$. 
This function can be expressed by 
\begin{align}	\label{APRDC}
M^{*}(x; q, a) = \frac{1}{\varphi(q)}\sum_{\chi \bmod{q}}\bar{\chi}(a)\psum_{n \leq x}\chi(n)\mu(n)
\end{align}
from the orthogonality of characters.
Here the first sum runs over all Dirichlet characters modulo $q$, 
and $\varphi$ is the Euler totient function.
Therefore, as the first step, we show the explicit formulas for the summatory function
\begin{align*}
M^{*}(x, \chi) = \psum_{n \leq x}\chi(n)\mu(n).
\end{align*}


\begin{theorem}	\label{exfoML1}
Let $x > 0$, $q \geq 2$, $T \geq \max\l\{T_0, \exp\l(q^{1 / 3} \r), 2 / x \r\}$ 
with $T_{0}$ a sufficiently large absolute constant.
Then, uniformly for all primitive Dirichlet characters $\chi$ modulo d with $d \leq q$, there exists a $T_{\nu} \in [T, 2T]$ satisfying
\begin{align*}
M^{*}(x, \chi) = 
&\sum_{|\gamma| < T_{\nu}}\frac{1}{(m(\rho)-1)!}\lim_{s \rightarrow \rho}\frac{d^{m(\rho)-1}}{ds^{m(\rho)-1}}
			\l(\frac{(s-\rho)^{m(\rho)}}{L(s, \chi)}\frac{x^{s}}{s}\r)\\
&+ \Res_{s = 0}\l(\frac{x^{s}}{L(s, \chi)s}\r) + \sum_{l = 1}^{\infty}\Res_{s = -l}\l(\frac{x^{s}}{L(s, \chi)s}\r) + R,
\end{align*}
where $L(s, \chi)$ is the Dirichlet $L$-function associated with $\chi$, 
and $m(\rho)$ is the multiplicity of the non-trivial zero $\rho$ of $L(s, \chi)$, and $R$ satisfies the estimate
\begin{align}	\label{APME}
R 
\ll &\frac{x}{T}\l(\log(x + 3) + \exp\l( C(\log{\log{T}})^2 \r)\r)+ \min\l\{1, \frac{x}{ T\l< x \r> } \r\}.
\end{align}
Here, $\l< x \r>$ denotes the distance from $x$ to nearest square-free integers coprime to $q$, other than $x$ itself.
Moreover, we find that if $\chi$ is an odd character $($i.e., $\chi(-1) = -1$$)$, then for $l \geq 1$,
\begin{align*}
&\Res_{s = -l}\l( \frac{x^{s}}{L(s, \chi)s} \r)\\
&= \l\{
\begin{array}{cl}
\d{\frac{(-1)^{k}i2(qx / 2 \pi)^{-(2k-1)}}{\tau(\chi)L(2k, \bar{\chi})(2 k - 1)(2k - 1)!}}
		& \text{if \; $l$ is odd with $l = 2k - 1$,}\vspace{2mm} \\
0
		& \text{if \; $l$ is even,}
\end{array}
\r.
\end{align*}
and
\begin{align*}
\Res_{s = 0}\l( \frac{x^{s}}{L(s, \chi)s} \r) = \frac{\pi i}{\tau(\chi)L(1, \bar{\chi})},
\end{align*}
and that if $\chi$ is an even character $($i.e., $\chi(-1) = 1)$, then
\begin{align*}
&\Res_{s = -l}\l( \frac{x^{s}}{L(s, \chi)s} \r)\\
&= \l\{
\begin{array}{cl}
\d{\frac{(-1)^{k}(qx / 2 \pi)^{-2k}}{\tau(\chi)L(2k + 1, \bar{\chi})k(2k)!}}	& \text{if \; $l$ is even with $l = 2k$,}\vspace{2mm} \\
0													& \text{if \; $l$ is odd,}
\end{array}
\r.
\end{align*}
and
\begin{align*}
\Res_{s = 0}\l( \frac{x^{s}}{L(s, \chi)s} \r) 
= \frac{2}{\tau(\chi)L(1, \bar{\chi})}\l( \log\l( \frac{q x}{2 \pi} \r) + \frac{L'(1, \bar{\chi})}{L(1, \bar{\chi})} - \gamma \r),
\end{align*}
where $\gamma$ is the Euler-Mascheroni constant, and $\tau(\chi)$ denotes the Gauss sum
\begin{align*}
\tau(\chi) = \sum_{a = 1}^{q}\chi(a)\exp\l( \frac{2\pi i a}{q} \r).
\end{align*}
\end{theorem}

This explicit formula is the case of primitive characters. 
On the other hand, for our purpose, we need the analogue of Theorem \ref{exfoML1} for imprimitive characters.
Here we can associate an imprimitive character $\chi$ with a primitive character $\chi^{*}$ inducing $\chi$ by the formula
\begin{align*}
L(s, \chi) = L(s, \chi^{*})\prod_{p | q}\l(1 - \frac{\chi^{*}(p)}{p^{s}}\r).
\end{align*}
Here we put
\begin{align}	\label{def_F}
F_{q, \chi^{*}}(s) := \prod_{p | q}\l(1 - \frac{\chi^{*}(p)}{p^{s}}\r).
\end{align}
In the following, we consider the case $F_{q, \chi^{*}} \not\equiv 1$.
Then this function $F_{q, \chi^{*}}$ has zeros only on imaginary axis.
In addition, from the uniqueness of the prime factorization, 
we can see that all zeros of $F_{q, \chi^{*}}$ are simple except the zero at $s = 0$.
Now, by studying $F_{q, \chi^{*}}$, we obtain an explicit formula for imprimitive characters as the following theorem.
Here we define the arithmetic function $\rad(n)$ by  
\begin{align*}
\rad(n) = \prod_{p | n}p.
\end{align*}

\begin{theorem}	\label{exfoML2}
Let $x > 0, q \geq 2, T \geq \max\l\{T_0, \exp\l(q^{1 / 3} \r), 2 / x \r\}$
with $T_{0}$ a sufficiently large absolute constant, $\chi$ be an imprimitive Dirichlet character modulo $q$, and 
$\chi^{*}$ be the primitive character inducing  $\chi$.
Then, uniformly for all imprimitive Dirichlet character modulo $q$ 
with $F_{q, \chi^{*}} \not\equiv 1$, there exist  
$T_{\nu} \in [T, 2T]$ and $T_{*} \in [T_{\nu}, T_{\nu} + 1]$ satisfying
\begin{align*}
M^{*}(x, \chi) =
&\sum_{|\gamma| < T_{\nu}}\frac{1}{(m(\rho)-1)!}\lim_{s \rightarrow \rho}\frac{d^{m(\rho)-1}}{ds^{m(\rho)-1}}
			\l(\frac{(s-\rho)^{m(\rho)}}{L(s, \chi)}\frac{x^{s}}{s}\r)
+\sum_{|\eta| < T_{*}}\Res_{s = i\eta}\l(\frac{x^{s}}{L(s, \chi)s}\r)\\
&+\sum_{l = 1}^{\infty}\frac{1}{F_{q, \chi^{*}}(-l)}\Res_{s = -l}\l(\frac{x^{s}}{L(s, \chi^{*})s}\r) + R,
\end{align*}
where the second sum on the right hand side runs over zeros of $L(s, \chi)$ on the imaginary axis, 
and $R$ is the error term satisfying estimate \eqref{APME}.
In addition, if $x \geq q^{c}\exp\l(c(\log{T})^{2/3}(\log{\log{T}})^{1/3}\r)$ holds for a sufficiently large constant $c$, then we have
\begin{align}	\label{IAzeros_bd}
\sum_{|\eta| < T_{*}}\Res_{s = i\eta}&\l(\frac{x^{s}}{L(s, \chi)s}\r)
=\frac{(\log{x})^{r + 1 - \kappa}}{L^{(r + 1 - \kappa)}(0, \chi)} + O_{q}((\log{x})^{r - \kappa})\\ \nonumber
&+ O\l( (\log{x})^{\omega\l(q' / b' \r)} \exp\l( C\sqrt{\omega\l(q' / b'  \r)\log\l(q' / b'\r)\log{\log{x}}} \r)\r).
\end{align}
Here, $b$ is the modulus of $\chi^{*}$, $q' = \rad(q)$, $b' = \rad(b)$, $C$ is a positive absolute constant, $\kappa = \kappa(\chi)$ denotes
\begin{align*}
\kappa = \l\{
\begin{array}{ll}
0	& \text{if \; $\chi$ is an even character,}\\
1	& \text{if \; $\chi$ is an odd character,}
\end{array}
\r.
\end{align*}
$\omega(q)$ is the number of distinct prime factors of $q$,
and $r$ indicates the number of the prime factors $p$ of $q$ with $\chi^{*}(p) = 1$.
\end{theorem}

From the above two theorems, we can obtain the explicit formula for 
$M^{*}(x; q, a)$ that is the our first purpose.
To abbreviate we define $L_{-1}(s; q, a)$ by
\begin{align*}
L_{-1}(s; q, a) := \frac{1}{\varphi(q)}\sum_{\chi \bmod{q}}\bar{\chi}(a)L(s, \chi)^{-1},
\end{align*}
where $a, q$ are positive integers with $(a, q) = 1$. 
The following corollary is our first main result in the present paper.

\begin{corollary}	\label{exfomLL}
Let $x > 0$, $T \geq \max\l\{T_0, \exp\l(q^{1 / 3} \r), 2 / x \r\}$
with $T_{0}$ a sufficiently large absolute constant, 
$a, q \in \ZZ_{>0}$ with $(a, q) = 1$. 
Then, there exist some $T_{\nu} \in [T, 2T]$ and $T_{*} \in [T_{\nu}, T_{\nu} + 1]$ satisfying
\begin{align*}
M^{*}(x; q, a) = 
&\sum_{|\gamma| < T_{\nu}}\frac{1}{(m(\rho)-1)!}\lim_{s \rightarrow \rho}\frac{d^{m(\rho)-1}}{ds^{m(\rho)-1}}
			\l((s-\rho)^{m(\rho)}L_{-1}(s; q, a)\frac{x^{s}}{s}\r)\\
&+\sum_{|\eta| < T_{*}}\Res_{s = i\eta}\l(L_{-1}(s; q, a)\frac{x^{s}}{s}\r)
+\sum_{l = 1}^{\infty}\Res_{s = -l}\l(L_{-1}(s; q, a)\frac{x^{s}}{s}\r) + R,
\end{align*}
where the first sum runs over non-trivial zeros $\rho = \beta + i\gamma$ of Dirichlet $L$-functions modulo $q$, 
and the second sum runs over zeros $i\eta$ of imprimitive Dirichlet $L$-functions on imaginary axis.
Futhermore, $R$ is the error term satisfying estimate \eqref{APME}, and the second series on the right hand side is estimated as
\begin{align*}
&\sum_{|\eta| < T_{*}}\Res_{s = i\eta}\l(L_{-1}(s; q, a)\frac{x^{s}}{s}\r)\\
&= \frac{(\log{x})^{\omega(q) + 1}}{\varphi(q)L^{(\omega(q) + 1)}(0, \chi_0)}
 + O\l(\varphi(q)^{-1}(\log{x})^{\omega(q)}\exp\l( C\sqrt{\omega(q)\log(q)\log{\log{x}}} \r) \r)\\
& \qquad + O_{q}\l((\log{x})^{\omega(q)}\r),
\end{align*}
where $\chi_0$ is the principal character modulo $q$, and $C$ is a positive absolute constant.
In particular, we have
\begin{align*}
&\sum_{|\eta| < T_{*}}\Res_{s = i\eta}\l(L_{-1}(s; q, a)\frac{x^{s}}{s}\r)
+\sum_{l = 1}^{\infty}\Res_{s = -l}\l(L_{-1}(s; q, a)\frac{x^{s}}{s}\r)\\
& \qquad \sim \frac{(\log{x})^{\omega(q) + 1}}{\varphi(q)L^{(\omega(q) + 1)}(0, \chi_0)} \qquad (x \rightarrow +\infty).
\end{align*}
\end{corollary}

To prove the above theorems, we need some upper bound of $1 / L(s, \chi)$ in certain domains, which is embodied in the following two propositions.

\begin{proposition}	\label{uniformlowerboundforL}
Let $\a \geq 13, T \geq T_{0}(\a) > 0$, and $1 \leq Q \leq (\log{T})^{\a / 4}$, where $T_{0}(\a)$ is a sufficiently large constant depending only on $\a$.
Then, we have
\begin{align*}
\min_{T \leq t \leq 2T}\l(\underset{\chi \in S(Q)}{\max_{\frac{1}{2} \leq \s \leq 2}}|L(\s + it, \chi)|^{-1}\r)
\leq \exp(C \a (\log{\log(QT)})^2),
\end{align*}
where $C$ is a positive absolute constant, and $S(Q)$ is the set of all primitive Dirichlet characters modulo $q$ with $q \leq Q$. 
\end{proposition}

This proposition is the consequence for primitive Dirichlet characters.
On the other hand, we need a similar result for imprimitive Dirichlet characters to prove  Theorem \ref{exfoML2}.
Then we need the upper bound of $1 / F_{q, \chi^{*}}$, which is the following proposition.

\begin{proposition}	\label{unibd_F}
Let $q \geq 2$ be an integer, $|T| \geq \omega(q)$, $S_1(q)$ be a nonempty subset of the set of all imprimitive Dirichlet characters modulo $q$ with $F_{q, \chi^{*}} \not\equiv 1$, 
and $d$ be the smallest modulus
of a primitive character $\chi^{*}$ inducing $\chi$ with $\chi \in S_{1}(q)$.
Then we have
\begin{align*}
&\min_{t \in [T, T + 1]}\l(\underset{\chi \in S_1(q)}
{\max_{|\s| \leq h}}|F_{q, \chi^{*}}(\s + it)|^{-1}\r)\\
&\leq \exp\l(C'' \omega\l(q' / d'\r)\log\l(\#S_{1}(q)\omega\l( q' / d' \r) + 2\r)
\l( 1 + \sqrt{\frac{\log\l(q' / d'\r) / \omega\l(q' / d'\r)}
{\log\l( \#S_{1}(q)\omega\l( q' / d'\r) + 2 \r)}} \r)\r),
\end{align*}
where 
$C$ is a sufficiently large positive absolute constant, $q' = \rad(q), d' = \rad(d)$, and 
$h \asymp \sqrt{\frac{\omega\l(q' / d'\r) / \log\l(q' / d'\r)}{\log\l( \#S_{1}(q)\omega\l( q' / d'\r) + 2 \r)}}$.
\end{proposition}
We are going to prove some properties of $F_{q, \chi^{*}}$ including this proposition in section \ref{FEPNC}. 

Here we note that Proposition \ref{uniformlowerboundforL} is an extension 
to Dirichlet $L$-functions of the result in the case of the Riemann zeta-function by
 Ramachandra and Sankaranarayanan \cite[Theorem 1.2]{RS}.
This result is useful when we prove some explicit formulas including the above formulas. 
For example, K\"uhn, Robles and Roy showed an explicit formula involving 
the M\"obius function and a primitive Dirichlet character 
under the Riemann Hypothesis and the simple zero conjecture 
for Dirichlet $L$-functions \cite[Theorem 1.1 (ii)]{KRR}. 
The author expects that it is possible to prove its explicit formula without
the Riemann Hypothesis for Dirichlet $L$-functions.
In fact, they use the Riemann Hypothesis for Dirichlet $L$-functions only in the proof 
of their Lemma 2.2, and in this paper, we are going to prove 
Proposition \ref{uniformlowerboundforL} which is an unconditional alternative of their Lemma 2.2.

One more useful point of this consequence is the uniformity for 
Dirichlet characters modulo $q$ with $q \leq Q$.
From this uniformity, there are some applications.
For example, one of the applications is that we can take $T_{\nu}$ not depending on the characters modulo $q$ in Corollary \ref{exfomLL}.
In addition, by the following result, we can apply Proposition \ref{uniformlowerboundforL} to a certain number field.

\begin{proposition}[Theorem 8.2 in \cite{WN}]	\label{relationwithLandD}
Let $K / \QQ$ be an Abelian number field, $K \subset K_{m}$ be the $m$-th cyclotomic field, and
$X(K)$ be the group of all characters $\Gal(K_{m} / \QQ)$ which are equal to unity on $\Gal(K_{m} / K)$.
Then we have
\begin{align*}
\zeta_{K}(s) = \prod_{\chi \in X(K)}L(s, \chi^{*}),
\end{align*}
where $\chi^{*}$ is the primitive Dirichlet character inducing $\chi$. 
\end{proposition}

The following corollary is an immediate consequence of Propositions \ref{uniformlowerboundforL} and \ref{relationwithLandD}.

\begin{corollary}	\label{abellowerbound}
Let $\a \geq 13$ and $T \geq T_{0}(\a) > 0$ with $T_{0}(\a)$ a sufficiently large constant depending only on $\a$.
If $K$ is an Abelian number field, $K_m$ is the smallest cyclotomic field satisfying $K \subset K_{m}$, then we have
\begin{align*}
\min_{T \leq t \leq 2T}\l(\max_{\frac{1}{2} \leq \s \leq 2}|\zeta_K(\s + it)|^{-1}\r)
\leq \exp(C \a (\#X(K)) (\log{\log(mT)})^2)
\end{align*}
for $m \leq (\log{T})^{\a / 4}$, where $C$ is a positive absolute constant.
\end{corollary}

\begin{remark}
Corollary \ref{abellowerbound} is a consequence for an Abelian number field.
On the other hand, probably, it is difficult to extend Corollary \ref{abellowerbound} 
to any number field.
The reason is that a zero density theorem for Dirichlet $L$-functions 
in the region close to critical line plays an important role in the proof 
of Proposition \ref{uniformlowerboundforL}, but it is difficult to obtain the zero density 
theorem of the same type for Dedekind zeta-functions.
\end{remark}

By Corollary \ref{abellowerbound}, we can obtain the explicit formula for the summatory function 
of the M\"obius function $\mu_{K}$ on an Abelian number field $K$. 
This M\"obius function $\mu_{K}$ is defined by
\begin{align*}
\mu_{K}(\mf{a}) = \l\{
\begin{array}{cl}
1		& \text{if \; $\mf{a} = O_{K}$,}\\
(-1)^{k}	& \text{if \; $\mf{a}$ is the product of $k$ distinct prime ideals,}\\
0		& \text{otherwise.}
\end{array}
\r.
\end{align*}
Here, we define the summatory function $M_{K}^{*}(x)$ by
\begin{align*}
M_{K}^{*}(x) = \psum_{N(\mf{a}) \leq x}\mu_{K}(\mf{a}),
\end{align*}
where the sum on the right hand side runs over integral ideals $\mf{a}$ of the ring $O_K$, and $N(\mf{a})$ is the absolute norm of $\mf{a}$, and
$\sum'$ indicates that if $x$ is an integer, then $\sum'_{N(\mf{a}) \leq x} = \sum_{N(\mf{a}) < x} + \frac{1}{2}\sum_{N(\mf{a}) = x}$.
Then we obtain the following theorem.

\begin{theorem}	\label{exfoMDA}
Let $K$ be an Abelian number field, $K_m$ be the smallest cyclotomic field satisfying $K \subset K_{m}$, 
$x > 0$, $T \geq \max\l\{T_0, \exp\l(m^{1 / 3} \r), 2 / x \r\}$
with $T_{0}$ a sufficiently large absolute constant.
Then, there exist some $T_{\nu} \in [T, 2T]$ satisfying
\begin{align*}
M_{K}^{*}(x) = 
&\sum_{|\gamma| < T_{\nu}}\frac{1}{(m(\rho)-1)!}\lim_{s \rightarrow \rho}\frac{d^{m(\rho)-1}}{ds^{m(\rho)-1}}
			\l((s-\rho)^{m(\rho)}\frac{x^{s}}{\zeta_{K}(s)s}\r)\\
&+\sum_{l=0}^{\infty}\Res_{s = -l}\l(\frac{x^{s}}{\zeta_K(s)s}\r) + R'.
\end{align*}
Here, $R'$ satisfies the estimate
\begin{align}	\label{R'_bound}
R'
\ll &\frac{x}{T}\min\l\{e^{n_K / x}(\log(x + 2))^{n_K}, \kappa_K \log(x + 2) + \frac{\Phi_{0}(K)}{(\log(x + 2))^{-1} + 1 / n_K}\r\} \\ \nonumber
&+ \exp\l( C n_K (\#X(K))(\log{\log{T}})^2 \r)
+ a_{n_x}\min\l\{1, \frac{x}{ T |x - n_x| } \r\},
\end{align}
where $n_K$ is the degree of $K$, $\kappa_K$ is the residue of $\zeta_K$ at $s = 1$, and 
$\Phi_0(K)$ is a constant depending only on K such that
\begin{align*}
\l| \sum_{N(\mf{a}) \leq x} 1 - \kappa_K x \r| \leq \Phi_{0}(K)x^{1 - 1 / n_K}.
\end{align*}
In addition, let $n_x$ denote one of the nearest positive integer from $x$ 
other than $x$ itself such that there exist ideals 
$\mf{a} \subset O_K$ with $N(\mf{a}) = n_x$ and $\mu_{K}(\mf{a}) \not= 0$.
If there exist several such integers, then we understand that $n_x$ is the one 
that $a_{n_{x}}$ is the biggest among them.
Moreover, we have
\begin{align*}
\Res_{s = 0}\l( \frac{x^{s}}{\zeta_K(s)s} \r)=
-\frac{2^{r_1 + r_2}\pi^{r_2} (\log{x})^{r_1 + r_2 - 1}}{ |d_K|^{1/2} \kappa_K} 
+ O_{K}\l( (1 - \delta_{0, r_1 + r_2 - 1})|\log{x}|^{r_1 + r_2 - 2} \r),
\end{align*}
and
\begin{align*}
\Res_{s = -l}\l( \frac{x^{s}}{\zeta_K(s)s} \r)
\ll \l\{
\begin{array}{ll}
\d{\frac{C^{n_K}x^{-l}}{l^{n_K / 2 + 1}}\l(\frac{2 \pi e}{l}\r)^{n_K l}(\log(x + 3))^{r_1 + r_2 - 1}}	&\text{if \; $l$ is even,}\\
\d{(1 - \delta_{0, r_2})\frac{C^{n_K}x^{-l}}{l^{n_K / 2 + 1}}\l(\frac{2 \pi e}{l}\r)^{n_K l}(\log(x + 3))^{r_2 - 1}}	&\text{if \; $l$ is odd}
\end{array}
\r.
\end{align*}
for $l \in \ZZ_{\geq 1}$, where $r_1$ is the number of real embeddings, $2r_2$ is the number of complex embeddings, 
and $\delta_{i, j}$ is the Kronecker delta.
In particular, we have
\begin{align*}
\sum_{l=0}^{\infty}\Res_{s = -l}\l(\frac{x^{s}}{\zeta_K(s)s}\r)
\sim -\frac{2^{r_1 + r_2}\pi^{r_2} (\log{x})^{r_1 + r_2 - 1}}{ |d_K|^{1/2} \kappa_K}  \quad (x \rightarrow +\infty).
\end{align*}
\end{theorem}

Here we do not consider refined upper bounds of $\Phi_{0}(K)$, but it is studied by Murty and Order in \cite{MOE}.

Moreover, as one more application of Proposition \ref{uniformlowerboundforL}, 
there are the following results for the sum involving derivative functions.
The following some theorems are the generalization of the result 
in the case of the Riemann zeta-function by Garaev and Sankaranarayanan \cite{GS}.

\begin{theorem}	\label{onemomentforLunconditional}
Let $\chi$ be a primitive Dirichlet character modulo $q$, and assume the simplicity of all  complex zeros of $L(s, \chi)$.
Then, for $T > \exp\l( q^{1/3} \r)$, there exist some $T_{\nu} \in [T, 2T]$ satisfying
\begin{align*}
\sum_{0 < \gamma < T_{\nu}}\frac{1}{L'(\rho, \chi)} = \frac{T_{\nu}}{2 \pi} + O\l(\exp\l(C (\log{\log{T}})^2 \r) + C(\chi)\r),
\end{align*}
where the sum on the left hand side runs over non-trivial zeros $\rho = \beta + i\gamma$ of $L(s, \chi)$, 
and $C(\chi)$ is a sufficiently large constant depending only on $\chi$.
Moreover, for any $T > T_{0}(q) > 0$ with a sufficiently large constant $T_{0}(q)$ depending only on $q$, we have
\begin{align}	\label{prGHHforL}
\sum_{0 < \gamma \leq T}\frac{1}{|L'(\rho, \chi)|} \gg T.
\end{align}
In particular, we also have
\begin{align}	\label{prGHHforLL}
\frac{1}{\varphi(q)}\sum_{\chi \bmod{q}}\sum_{0 < \gamma \leq T}\frac{1}{|L'(\rho, \chi)|}
\gg T.
\end{align}
\end{theorem}

We can obtain a theorem for $\zeta_K(s)$ with an Abelian number field $K$, 
similar to Theorem \ref{onemomentforLunconditional}.
It is the following result.

\begin{theorem}	\label{onemomentforADunconditional}
Let $K$ be an Abelian number field, $K_m$ be the smallest cyclotomic field satisfying $K \subset K_{m}$, 
and assume the simplicity of all complex zeros of $\zeta_K(s)$. 
Then, for $T > \exp\l( m^{1/3} \r)$, there exist some $T_{\nu} \in [T, 2T]$ satisfying
\begin{align*}
\sum_{0 < \gamma < T_{\nu}}\frac{1}{\zeta_{K}'(\rho)} 
= \frac{T_{\nu}}{2 \pi} + O\l(\exp\l(C (\#X(K)) (\log{\log{T}})^2 \r) + C(K)\r),
\end{align*}
where $X(K)$ is the same as in Proposition \ref{relationwithLandD}, 
the sum on the left hand side runs over non-trivial zeros $\rho = \beta + i\gamma$ of $L(s, \chi)$, 
and $C(K)$ is a sufficiently large constant depending only on $K$.
In particular, for any $T \geq T_{0}(K) > 0$ with a sufficiently large constant $T_{0}(K)$ depending only on $K$, we have
\begin{align}	\label{prGHHforD}
\sum_{0 < \gamma \leq T}\frac{1}{|\zeta_{K}'(\rho)|} \gg T.
\end{align}
\end{theorem}

We omit the proof of Theorem \ref{onemomentforADunconditional} because 
the proof is almost the same as the proof of Theorem \ref{onemomentforLunconditional} 
by using Corollary \ref{abellowerbound}.

Here we mention some comments on \eqref{prGHHforLL} and \eqref{prGHHforD}.  
These results are useful when we consider the exact behavior of some partial sum of some
M\"obius functions under a conjecture.
In fact, estimates \eqref{prGHHforLL} and \eqref{prGHHforD} can be applied to show 
some $\Omega$-results for certain summatory functions of the M\"obius functions 
under the Linear Independence Conjecture.
Here the Linear Independence Conjecture is the following conjecture.

\begin{conjecture*}[Linear Independence Conjecture for Dedekind zeta-functions (cf. \cite{CRH})]
The positive imaginary parts of the zeros of any Dedekind zeta-function are linearly independent over $\QQ$.
\end{conjecture*}

Note that the Linear Independence Conjecture for Dedekind zeta-functions implies 
the same type conjecture for Dirichlet $L$-functions by Proposition \ref{relationwithLandD}.

Now by estimates \eqref{prGHHforLL} and \eqref{prGHHforD}, as extension of Ingham's theorem \cite{In}, we can obtain that, for $(a, q) = 1$,
\begin{align*}
\underset{x \rightarrow \infty}{\underline{\overline{\lim}}}\frac{M(x; q, a)}{x^{1/2}} = \pm \infty,
\end{align*}
and, for any Abelian number field $K$, 
\begin{align}	\label{EID}
\underset{x \rightarrow \infty}{\underline{\overline{\lim}}}\frac{M_{K}(x)}{x^{1/2}} = \pm \infty
\end{align}
under the Linear Independence Conjecture for Dedekind zeta-functions. These proofs are similar to the proof of Corollary 15.7 in \cite{MV}.
In addition, we can remove the condition ``Abelian'' in \eqref{EID} by assuming 
the Riemann Hypothesis for Dedekind zeta-functions.

\begin{remark}
We can generalize the above two theorems to the statement which is analogous
to the Landau-Gonek formula (cf. \cite[Proposition 2]{BHMS} and \cite[Theorem 1]{GL}), 
i.e. for some $T_{\nu} \in [T, 2T]$,
\begin{align}	\label{ALGF}
\sum_{0 < \gamma < T_{\nu}}\frac{x^{\rho}}{L'(\rho, \chi)}, \sum_{0 < \gamma < T_{\nu}}\frac{x^{\rho}}{\zeta_K'(\rho)}
\end{align}
are estimated by a little modified asymptotic formula with the original Landau-Gonek formula under the simple zero conjecture for the corresponding function.
Moreover, if the Riemann Hypothesis for the corresponding function 
$F(s) (=L(s, \chi) \; \mathrm{or} \; \zeta_K(s))$ and $|F(\rho)|^{-1} \ll |\rho|^{1 - \delta}$ for some fixed constant $\delta > 0$ are also true, 
then we have an analogue of the Landau-Gonek formula for \eqref{ALGF} for any sufficiently large $T > 0$. 
We only mention this fact here because 
the author cannot find some useful applications of these consequences.
\end{remark}

Here the author raises the following conjecture suggested by the above results.

\begin{conjecture}
Let $\chi$ be a primitive Dirichlet character, and $K$ be an Abelian number field.
Then
\begin{align*}
\sum_{0 < \gamma \leq T}\frac{1}{L'(\rho, \chi)}, \sum_{0 < \gamma \leq T}\frac{1}{\zeta_{K}'(\rho)}
\sim \frac{T}{2\pi} \qquad (T \rightarrow \infty).
\end{align*}
\end{conjecture}

We can prove this conjecture in the case of the Riemann zeta-function under some known conjectures 
that are the Riemann Hypothesis, the simple zero conjecture and the estimate $|\zeta'(\rho)|^{-1} \ll |\rho|^{1/3 + \epsilon}$.
In fact, we can obtain the following asymptotic formula
\begin{align*}
\sum_{0 < \gamma \leq T}\frac{1}{\zeta'(\rho)} = \frac{T}{2 \pi} + O_{\epsilon}\l(T^{1/3 + \epsilon}\r)
\end{align*}
under these conjectures.
The present paper does not give the proof of this estimate because it is almost similar to the proof of Theorem 15.6 in \cite{MV}.

\section{\textbf{On estimates of Dirichlet $L$-functions in certain domains
}}


In this section, we are going to show some estimates of Dirichlet $L$-functions including Proposition \ref{uniformlowerboundforL}.
Firstly, we refer to an important result on the zero density theorem for Dirichlet $L$-functions by Montgomery.

\begin{lemma}[Theorem 1 in \cite{MZD}] \label{ZDBM}
Let $S(Q)$ denote the set of all primitive Dirichlet characters modulo $q$ with $q \leq Q$.
For $Q \geq 1, T \geq 2$, and $\frac{1}{2} \leq \s \leq 1$, we have
\begin{align*}
\sum_{\chi \in S(Q)}N_{\chi}(\s, T)
\ll (Q^2 T)^{\frac{3(1 - \s)}{2- \s}}(\log(QT))^{13},
\end{align*}
where $\sum_{\chi}^{*}$ denotes a sum over all primitive Dirichlet characters modulo $q$, 
and $N_{\chi}(\s, T)$ is the number of zeros $\rho$ of $L(s, \chi)$ with $\Re(\rho) \geq \s$ and $|\Im(\rho)| \leq T$.
\end{lemma}

By using this lemma, we show the following proposition.

\begin{proposition} \label{main_proposition}
Let $\a \geq 13, T \geq T_{0}(\a) > 0$ with $T_{0}(\a)$ a sufficiently large number depending 
only on $\a$, and $1 \leq Q \leq (\log{T})^{\a / 4}$.
Then there exists a closed interval $J_{0}$ of length $(\log(QT))^{\a / 3}$ contained in $[T, 2T]$ such that
\begin{align*}
\underset{\chi \in S(Q)}{\max_{\s \geq 1/2 + 14 \a r, t \in J_0}}\l|\log{L(s, \chi)}\r| 
\ll \a \log{\log(QT)},
\end{align*}
where $r = (\log{\log(QT)})^2(\log(Q^2 T))^{-1}$.
\end{proposition}

\begin{proof}
Let $D = \frac{\a}{3}$ and $I_{j} = \l[T + 2(j - 1)(\log(QT))^D, T + 2j(\log(QT))^{D}\r)$. 
By Lemma \ref{ZDBM}, if $\s \geq \frac{1}{2} + \a(\log{\log(QT)})(\log(Q^2T))^{-1}$, then
\begin{align} \label{ZDR}
\sum_{\chi \in S(Q)}N_{\chi}(\s, 2T) - \sum_{\chi \in S(Q)}N_{\chi}(\s, T)
\leq CQ^2T(\log(QT))^{-\a}.
\end{align}
Here $C$ is a sufficiently large absolute constant.
Now we consider the disjoint rectangles 
\begin{align*}
(\s, t) \in R_{j} =  \l[\frac{1}{2} + \a(\log{\log(QT)})(\log(Q^2 T))^{-1}, 2\r] \times I_{j},
\end{align*}
which we may regard subsets in the complex plane.
The number of these rectangles is $N = \l[\frac{1}{2}T(\log(QT))^{-D} \r]$.
By inequality \eqref{ZDR}, if $Q \leq (\log{T})^{\a / 4}$, then the number of zeros of Dirichlet 
$L$-functions attached to primitive characters modulo $q$ with $q \leq Q$ in the rectangle
 $R = \bigsqcup_{j = 1}^{N}R_{j}$ is less than $CT(\log(QT))^{-\a / 2}$. 
Therefore, if $T \geq T_{0}(\a)$ for sufficiently large number $T_{0}(\a)$ depending only
on $\a$, then the number of rectangles $R_{j}$ not having zeros of the Dirichlet $L$-functions
is greater than $N - CT(\log(QT))^{-\a} (\geq N/2)$.

Let $J$ be the set of all $j$ such that $R_{j}$ does not include zeros 
of those Dirichlet $L$-functions.
By using the Euler product for $L(s, \chi)$ and the Taylor expansion, for $\s > 1$, we find that
\begin{align}	\label{def_P_l}
\log{L(s, \chi)} 
&= \sum_{p}\chi(p)p^{-s} + \frac{1}{2}\sum_{p}\chi(p)^{2}p^{-2s} + \sum_{p}\sum_{n = 3}^{\infty}\frac{\chi(p)^{n}}{np^{ns}} \nonumber\\
&=: P_{1}(s, \chi) + \frac{1}{2}P_{2}(s, \chi) + \Psi(s, \chi),
\end{align}
and that $P_{2}(s, \chi)$ is regular on $\s > 1/2$, and $\Psi(s, \chi)$ is regular on $\s > 1/3$ and bounded on $\s \geq 1/2$.
In addition, $\log{L(s, \chi)}$ is regular on $R_{j} (j \in J)$ since $L(s, \chi)$ does not have zero on the same domain.
Hence $P_{1}(s, \chi)$ is analytically continued to $R_{j} (j \in J)$.

Let $k$ be a positive integer. We define $a_{k, \chi}$ by 
\begin{align*}
\l(P_1(s, \chi) + \frac{1}{2}P_2(s, \chi)\r)^{k} = \sum_{n = 1}^{\infty}\frac{a_{k, \chi}(n)}{n^{s}}.
\end{align*}
We can estimate $|a_{k, \chi}(n)|$ by the following way.
If $a_{k, \chi}(n) \not= 0$, then $n$ can be written in the form 
$n = p_{1}^{l_1} \cdots p_{k}^{l_k}$ ($l_i \in \{ 1, 2 \}$).
The number of ways which one can express $n$ by ordering $p_{1}^{l_1}, \dots, p_{k}^{l_k}$
in different ways is at most $k!$.
This means that we have the inequality
\begin{align*}
|a_{k, \chi}(n)| \leq k! \leq k^{k}.
\end{align*}
Hence, by the boundedness of $a_{k, \chi}(n)$ with respect to $n$, 
\begin{align*}
\sum_{n = 1}^{\infty}\frac{a_{k, \chi}(n)}{n^s}\exp\l(- \frac{n}{X} \r)
\end{align*}
is an entire function for any $X > 0$.
Here, as our first step, we show that if $X = \l(Q^{2} T\r)^{1 / 4}$, and 
$k = [\a \log{\log(QT)}]$, then we have
\begin{align} \label{lemoflem1}
\l(P_{1}(s, \chi) + \frac{1}{2}P_{2}(s, \chi)\r)^{k} = \sum_{n = 1}^{\infty}\frac{a_{k, \chi}(n)}{n^{s}}\exp\l( -\frac{n}{X} \r) + O(1)
\end{align}
for any $j \in J$ and $(\s, t)\in R_j'$, where $R_{j}' = [\frac{1}{2} + 14 k(\log{\log(QT)})(\log(Q^2T))^{-1}, 2] \times I_{j}'$ 
with $I_{j}' = [T + 2(j-1)(\log(QT))^{D} + (\log(QT))^2, T + 2j(\log(QT))^{D} - (\log(QT))^2]$.

The proof is as follows.
Using a formula for Abelian weight (cf.\ (5.25) in \cite{MV}) 
and the Stirling formula, if $K = (\log(QT))^2$, then we have
\begin{align*}
&\sum_{n = 1}^{\infty}\frac{a_{k, \chi}(n)}{n^{s}}\exp\l( -\frac{n}{X} \r)
= \frac{1}{2\pi i}\int_{2 - i\infty}^{2 + i\infty}\l( P_1(s + w, \chi) + \frac{1}{2}P_2(s + w, \chi) \r)^{k}\Gamma(w)X^{w}dw\\
&= \frac{1}{2\pi i}\int_{2 - iK}^{2 + iK}\l( P_1(s + w, \chi) + \frac{1}{2}P_2(s + w, \chi) \r)^{k}\Gamma(w)X^{w}dw\\
&\qquad + O\l( k^{k}(\log(QT))^{3}e^{-\pi (\log(QT))^2 / 2}X^2 \r).
\end{align*}
Here we consider the estimate of the integral on the right-hand side.
By the Borel-Carath\'eodry lemma, we can find that
\begin{align*}
|\log{L(s, \chi)}| \leq \frac{2|s - a|}{R - |s - a|}\max_{|s - a| = R}\Re(\log{L(s, \chi)}) + \frac{R + |s - a|}{R - |s - a|}|\log{L(a, \chi)}| 
\end{align*}
holds for $|s - a| < R$ with $R = 2, a = \frac{5}{2} + k(\log{\log(QT)})(\log(Q^2 T))^{-1} + it, t \in {I_{j}}'$.
Hence we have
\begin{align*}
|\log{L(s, \chi)}| \ll (\log(QT))^2(\log{\log(QT)})^{-2}
\end{align*}
for $\s \geq \frac{1}{2} + 2\a k(\log{\log(QT)})(\log(Q^2 T))^{-1}, t \in {I_{j}}'$ since $L(s, \chi) \ll Q|t|$ holds for $\s \geq 1/2$.
By this estimate and the boundedness of $\Psi(s, \chi)$ for $\s \geq 1/2$, we have
\begin{align}	\label{IE1}
\l| P_1(s, \chi) + \frac{1}{2}P_2(s, \chi) \r|^{k} 
&= \l| \log{L(s, \chi)} - \Psi(s, \chi) \r|^{k} \nonumber\\
&\leq C^{k}(\log(QT))^{2k}(\log{\log(QT)})^{-k}
\end{align}
for $\s \geq \frac{1}{2} + 2\a k(\log{\log(QT)})(\log(Q^2 T))^{-1}, t \in {I_{j}}'$.
In addition, by the residue theorem, if $\beta = 12 k(\log{\log(QT)})(\log(Q^2 T))^{-1}$, then
\begin{align*}
&\frac{1}{2\pi i}\int_{2 - iK}^{2 + iK}\l( P_1(s + w, \chi) + \frac{1}{2}P_2(s + w, \chi) \r)^{k}\Gamma(w)X^{w}dw = \\
&\frac{1}{2\pi i}\l( \int_{-\beta + iK}^{2 + iK} + \int_{-\beta - iK}^{-\beta + iK} + \int_{2 - iK}^{-\beta - iK} \r)
\l( P_1(s + w, \chi) + \frac{1}{2}P_2(s + w, \chi) \r)^{k}\Gamma(w)X^{w}dw\\
& \qquad + \l( P_{1}(s, \chi) + \frac{1}{2}P_{2}(s, \chi) \r)^k
\end{align*}
holds for $\s \geq \frac{1}{2} + 14 k (\log{\log(QT)})(\log(Q^2 T))^{-1}, t \in I_{j}'$. 
By inequality \eqref{IE1} and the Stirling formula, we have
\begin{align*}
&\frac{1}{2\pi i}\int_{-\beta \pm iK}^{2 \pm iK}
\l( P_1(s + w, \chi) + \frac{1}{2}P_2(s + w, \chi) \r)^{k}\Gamma(w)X^{w}dw\\
&\ll C^{k}(\log(QT))^{2k+3}(\log{\log(QT)})^{-2k}e^{-\frac{\pi}{2}(\log(QT))^{2}}X^2, 
\end{align*}
and
\begin{align*}
&\frac{1}{2\pi i}\int_{-\beta - iK}^{-\beta + iK}\l( P_1(s + w, \chi) + \frac{1}{2}P_2(s + w, \chi) \r)^{k}\Gamma(w)X^{w}dw\\
&\ll \beta^{-1}X^{-\beta}C^{k}(\log(QT))^{2k}(\log{\log(QT)})^{-2k}.
\end{align*}
From the above estimates, if $X = (Q^2T)^{1 / 4}$, and $k = [\a \log{\log(QT)}]$, then we obtain the formula \eqref{lemoflem1}.

Next we consider the function
\begin{align*}
F_{2k}(T, \chi) := \sum_{j \in J}\max_{s \in {R_j}'}|\log{L(s, \chi)}|^{2k}.
\end{align*}
By \eqref{def_P_l} and \eqref{lemoflem1}, we have
\begin{align*}
|\log{L(s, \chi)}|^{2k} 
&\leq 2^{2k}\l( |P_1(s, \chi) + \frac{1}{2}P_2(s, \chi)|^{2k} + |\Psi(s, \chi)|^{2k} \r)\\
&\leq 4^{2k}\l| \sum_{n = 1}^{\infty}\frac{a_{k, \chi}(n)}{n^{s}}\exp\l( -\frac{n}{X} \r) \r|^2 + C^{2k}.
\end{align*}
Therefore, if $r = (\log{\log(QT)})^{2}(\log(Q^2 T))^{-1}$, and $s_j$ is an element of $R_{j}'$ satisfying
\begin{align*}
\max_{s \in R_j'}|\log{L(s, \chi)}|^{2k} = |\log{L(s_j, \chi)}|^{2k},
\end{align*}
then we have
\begin{align*}
|\log{L(s_j, \chi)}|^{2k}
&\ll 4^{2k}\l| \sum_{n = 1}^{\infty}\frac{a_{k, \chi}(n)}{n^{s_{j}}}\exp\l( -\frac{n}{X} \r) \r|^2 + C^{2k}\\
&= \frac{4^{2k}}{\pi r^{2}}\l| \int \int_{|s - s_j| \leq r}\l(\sum_{n = 1}^{\infty}\frac{a_{k, \chi}(n)}{n^{s}}\exp\l( -\frac{n}{X} \r)\r)^2d\s dt \r|
+ C^{2k}\\
&\ll \frac{4^{2k}}{r^{2}}\int \int_{|s - s_j| \leq r}\l| 
\sum_{n = 1}^{\infty}\frac{a_{k, \chi}(n)}{n^{s}}\exp\l( -\frac{n}{X} \r) \r|^2 d\s dt + C^{2k}
\end{align*}
by the mean value theorem on analytic functions.
By the disjointness of the domains $|s - s_j| \leq r$ for each $j$ and the estimate
\begin{align*}
&\l|\sum_{n > X^2}\frac{a_{k, \chi}(n)}{n^{s}}\exp\l( -\frac{n}{X} \r)\r|
\leq k^k\sum_{n > X^2}\frac{1}{n^{1/2}}\exp\l( -\frac{n}{X} \r) \\
&= k^k\sum_{m = 0}^{\infty}\sum_{2^{m}X^2 < n \leq 2^{m+1}X^2}\frac{1}{n^{1/2}}\exp\l( -\frac{n}{X} \r)
\leq k^k \sum_{m = 0}^{\infty}2^{\frac{m}{2}}X\exp\l( -2^{m}X \r)\\
&\leq k^k X\l( \exp(-X) + \sum_{m = 1}^{\infty}\l(2^{1/2}e^{-X} \r)^{m} \r)
\ll k^{k}X\exp(-X),
\end{align*}
we have
\begin{align*}
F_{2k}(T, \chi) 
\ll &\frac{4^{2k}}{r^2}\int\int_{E}\l| \sum_{n \leq X^2}\frac{a_{k, \chi}(n)}{n^{s}}\exp\l( -\frac{n}{X} \r) \r|^2d \s d t \\
&+ \l(C^{2k} + r^{-2}(4k)^{2k} X^2 \exp(-2X) \r)N,
\end{align*}
where $E$ is the domain with $\frac{1}{2} \leq \s \leq 2, T \leq t \leq 2T$.
As for the remaining integral, we have
\begin{align*}
&\l|\sum_{n \leq X^2}\frac{a_{k, \chi}(n)}{n^{s}}\exp\l( -\frac{n}{X} \r) \r|^2
= \sum_{m, n \leq X^2}\frac{a_{k, \chi}(m)\bar{a_{k, \chi}(n)}}{m^{\s+it}n^{\s-it}}\exp\l( -\frac{m+n}{X} \r)\\
&=\l(\underset{m \not= n}{\sum_{m, n \leq X^2}} + \underset{m = n}{\sum_{m, n \leq X^2}}\r)
\frac{a_{k, \chi}(m)\bar{a_{k, \chi}(n)}}{m^{\s+it}n^{\s-it}}\exp\l( -\frac{m+n}{X} \r)\\
&=\underset{m \not= n}{\sum_{m, n \leq X^2}}\frac{a_{k, \chi}(m)\bar{a_{k, \chi}(n)}}{m^{\s+it}n^{\s-it}}\exp\l( -\frac{m+n}{X} \r)
+O\l(k^{2k}\sum_{n \leq X^2}\frac{1}{n^{2\s}} \r),
\end{align*}
and
\begin{align*}
&\int\int_{E}\l| \sum_{n \leq X^2}\frac{a_{k, \chi}(n)}{n^{s}}\exp\l( -\frac{n}{X} \r) \r|^2ds\\
&=\underset{m \not= n}{\sum_{m, n \leq X^2}}a_{k, \chi}(m)\bar{a_{k, \chi}(n)}\exp\l( -\frac{m+n}{X} \r)
\l(\int_{\frac{1}{2}}^{2}\frac{d\s}{(mn)^{\s}}\r)\l( \int_{T}^{2T}\l( \frac{n}{m} \r)^{it}dt \r)\\
& \qquad + O\l( k^{2k} Q^2T \log{X} \r)\\
& \ll k^{2k}\underset{m \not= n}{\sum_{m, n \leq X^2}}\l| \log\l( \frac{n}{m} \r) \r|^{-1}\frac{1}{(mn)^{1/2}}
+ k^{2k} Q^2 T \log{X} 
\ll k^{2k}X^4 + k^{2k} Q^2 T \log{X}.
\end{align*}
Hence we find that
\begin{align*}
\int\int_{E}\l| \sum_{n \leq X^2}\frac{a_{k, \chi}(n)}{n^{s}}\exp\l( -\frac{n}{X} \r) \r|^2ds
\ll k^{2k} Q^2 T \log{X}
\end{align*}
by $X = \l(Q^2T\r)^{1/4}$.
Thus we have
\begin{align*}
F_{2k}(T, \chi) 
\ll r^{-2}(4k)^{2k} Q^2 T\log{X} + \l(C^{2k} + r^{-2}(4k)^{2k} X^2 \exp(-2X) \r)N,
\end{align*}
and there exists a $j_0 \in J$ satisfying
\begin{align*}
\max_{s \in R_{j_0}'}|\log{L(s, \chi)}|^{2k} 
\ll 	& r^{-2} (4k)^{2k} (\log{X}) (\log(QT))^{\a/3}\\
	&+ C^{2k} + r^{-2}(4k)^{2k} X^2 \exp(-2X)
\end{align*}
by the definition of $F$.

From the above discussion, the inequality
\begin{align*}
\max_{s \in R_{j_0}'}|\log{L(s, \chi)}| 
\ll k \ll \a \log{\log(QT)}
\end{align*}
holds uniformly for $\chi \in S(Q)$, which completes the proof of Proposition \ref{main_proposition}.
\end{proof}

The following corollary is an immediate consequence of Proposition \ref{main_proposition}.

\begin{corollary}	\label{uniesconL}
We have
\begin{align}	\label{estimate_1}
\underset{\chi \in S(Q)}{\max_{\s \geq 1/2 + 14 \a r, t \in J_0}}\l|L(s, \chi)\r|^{\pm} 
\leq \exp\l(C \a \log{\log(QT)}\r) = (\log(QT))^{C \a},
\end{align}
where $C$ is a positive absolute constant, and the meaning of the other letters 
appearing in the formula is the same as in Proposition \ref{main_proposition}.
\end{corollary}

\begin{lemma}	\label{RLO}
Let $\chi$ be a primitive Dirichlet character modulo $q$. 
If $|t| \geq 1$ and $\s \leq \frac{1}{2}$, then
\begin{align}	\label{RLO1}
|L(s, \chi)| 
\asymp \l( \frac{2\pi e}{q|s|} \r)^{\s}(q|s|)^{1/2}\exp\l( |t|\tan^{-1}\l( \frac{1 - \s}{|t|} \r) \r)|L(1 - s, \bar{\chi})|.
\end{align}
If $|t| \leq 1, \s = -(m + 1 / 2)$, then
\begin{align}	\label{RLO2}
|L(s, \chi)| 
\asymp \l( \frac{2\pi e}{q|s|} \r)^{\s}(q|\s|)^{1/2}|L(1 - s, \bar{\chi})|.
\end{align}
\end{lemma}

\begin{proof}
By the functional equation for Dirichlet $L$-functions and the Stirling formula, we obtain this lemma.
\end{proof}

\begin{lemma} \label{Rama-Sanka-lemma2}
Let $J_1 = [y_1, y_2]$ be the closed interval that is obtained by removing intervals of length $\log(QT)$ from both ends of $J_{0}$.
Then we have
\begin{align*}
\underset{\chi \in S(Q)}{\max_{\s \geq 1/2 - 26\a r, t \in J_1}}\l| L(s, \chi) \r| 
\leq \exp\l( C \a (\log{\log(QT)})^2 \r),
\end{align*}
where $C$ is a positive absolute constant, and 
the meaning of the letters appearing in the formula 
is the same as in Proposition \ref{main_proposition}.
\end{lemma}

\begin{proof}
For $\s \geq 1 / 2 + 14 \a r, t \in J_{0}$, we have
\begin{align*}
|\log{L(s, \bar{\chi})}| 
&\geq \log|L(s, \bar{\chi})|
\asymp \log|(q t)^{1/ 2 - \s}L(1 - s, \chi)|\\
&\geq \log|L(1 - s, \chi)| - \l(\s -\frac{1}{2} \r)\log(qt)
\end{align*}
since $|L(s, \chi)| \asymp (q t)^{1/2 - \s}|L(1 - s, \bar{\chi})|$ by Lemma \ref{RLO}.
In addition, if $\chi$ is a primitive character, then $\bar{\chi}$ is also a primitive character. Therefore, for $t \in J_{0}$, we have
\begin{align*}
\log\l|L\l( \frac{1}{2} - 26\a r + it, \chi \r)\r|
&\ll \a r \log(QT) + \l|\log\l(L\l(\frac{1}{2} + 26\a r + it, \bar{\chi} \r)\r)\r|\\
&\ll \a (\log{\log(QT)})^2
\end{align*}
by Proposition \ref{main_proposition}.
Hence, if $t \in J_{0}$, then
\begin{align} \label{estimate_2}
L\l( \frac{1}{2} - 26\a r + it, \chi \r)
\ll \exp \l(C \a (\log{\log(QT)})^2\r).
\end{align}

Next we consider the function
\begin{align*}
g_{l}(w) = L(s_l + w, \chi)e^{w^2} \quad (l = 1, 2),
\end{align*}
where $s_l = \frac{1}{2} + iy_{l}$.
By a basic upper bound $L(s, \chi) \ll q (|t| + 1)$ for $\s \geq 1 / 4$, we have
\begin{align*}
g_{l}\l(x \pm i\log(QT)\r)
&\ll e^{x^2 - (\log(QT))^2}|L(s_l + x \pm i\log(QT), \chi)|\\
&\ll e^{x^2 -(\log(QT))^2}QT 
\ll 1
\end{align*}
for $- 26\a r \leq x \leq 26\a r, T \geq T_{0}(\a)$.
Moreover, by estimates \eqref{estimate_1} and \eqref{estimate_2}, we also have
\begin{align*}
g_{l}\l(\frac{1}{2} \pm 26\a r + iy\r) 
&= L\l(s_l \pm 26\a r + iy, \chi \r)e^{(\frac{1}{2} \pm 26\a r + iy)^2}\\
&\ll \exp(C \a (\log{\log(QT)})^2)
\end{align*}
for $-\log(QT) \leq y \leq \log(QT)$.
Hence, by the maximum modulus principle, we obtain
\begin{align*}
g_{l}(x + iy) \ll \exp(C \a (\log{\log(QT)})^2)
\end{align*}
for $-26\a r \leq x \leq 26\a r$ and $-\log(QT) \leq y \leq \log(QT)$.
In particular, if $y = 0$, then
\begin{align*}
L(s_l + x, \chi) 
&= g_{l}(x)e^{x^2}
\ll \exp(C \a (\log{\log(QT)})^2 + x^2)\\
&\ll \exp(C \a (\log{\log(QT)})^2)
\end{align*}
holds for $- 26\a r \leq x \leq 26\a r$.
Again by using the maximum modulus principle, we can find that
\begin{align*}
L(s, \chi) \ll \exp(C \a (\log{\log(QT)})^2)
\end{align*}
in the compact set $\frac{1}{2} - 26\a r \leq \s \leq \frac{1}{2} + 26\a r, t \in J_1$.
\end{proof}

\begin{lemma} \label{Landau_approximatio}
If $f$ is a regular function and
\begin{align*}
\l| \frac{f(s)}{f(s_0)} \r| < e^{M}, \quad (M > 1)
\end{align*}
in $|s - s_0| \leq r$, then for any constant $0 < \epsilon < \frac{1}{2}$, 
\begin{align*}
\frac{f'}{f}(s) = \sum_{|\rho - s_0| \leq \frac{1}{2}r}\frac{1}{s - \rho} + O_{\epsilon}\l( \frac{M}{r} \r)
\end{align*}
in $|s - s_0| \leq \l( \frac{1}{2} - \epsilon \r)r$, where $\rho$ is a zero of $f$.
\end{lemma}

\begin{proof}
This is Lemma 3 in \cite{RS}.
\end{proof}

\begin{proposition}	\label{main_proposition_2}
Let $J_2$ be the closed interval that is obtained by removing intervals of length $40 \a r$ 
from both ends of $J_1$.
If $s_0 = \frac{1}{2} + 14 \a r + it_0, t_0 \in J_2$, then we have
\begin{align}	\label{mpro2}
\frac{L'(s, \chi)}{L(s, \chi)} = \sum_{|\rho - s_0| \leq 20 \a r}\frac{1}{s - \rho} + O\l(\log(QT) \r)
\end{align}
for $|s - s_0| \leq 15 \a r$, where
the meaning of the letters appearing in the following formula 
is the same as the above situations. 
\end{proposition}

\begin{proof}
By Corollary \ref{uniesconL}, we have
\begin{align*}
|L(s, \chi)| \geq \exp\l( -C \a \log{\log(QT)} \r)
\end{align*}
for $\s \geq \frac{1}{2} + 14 \a r, t \in J_0$.
By this inequality and Lemma \ref{Rama-Sanka-lemma2}, we find that
\begin{align*}
\l|\frac{L(s, \chi)}{L(s_0, \chi)}\r|
\leq \exp\l( C \a(\log{\log(QT)})^2 \r)
\end{align*}
for $|s - s_0| \leq 40 \a r$.
Hence, by Lemma \ref{Landau_approximatio}, if $\epsilon = \frac{1}{8}$, then we obtain
\begin{align*}
\frac{L'(s, \chi)}{L(s, \chi)} 
&= \sum_{|\rho - s_0| \leq 20 \a r}\frac{1}{s - \rho} + O\l( (4 \a r)^{-1} C \a (\log{\log(QT)})^2 \r)\\
&= \sum_{|\rho - s_0| \leq 20 \a r}\frac{1}{s - \rho} + O\l(\log(QT) \r)
\end{align*}
for $|s - s_0| \leq 15 \a r$.
\end{proof}

\begin{lemma} \label{Rama-Sanka-Lemma5}
Let $t_0 \in J_{2}$.
If $s_0 = \frac{1}{2} + 14 \a r + it_{0}$, then
\begin{align*}
\log{L(\s + it_0, \chi)}
=&\sum_{|\rho - s_0| \leq 20 \a r}\l( \log(\s + it_0 - \rho) - \log\l(s_0 - \rho \r) \r)\\
&+ O\l( \a (\log{\log(QT)})^2\r)
\end{align*}
for $\frac{1}{2} - \a r \leq \s \leq \frac{1}{2} + 29 \a r$.
In particular, by taking the real parts in the both sides, we find that
\begin{align}	\label{Rama-Sanka-Lemma5-2}
\log\l|L(\s + it_0, \chi)\r| = \sum_{|\rho - \s_0| \leq 20 \a r}\log\l|\frac{\s + it_0 - \rho}{s_0 - \rho}\r| 
+ O\l( \a (\log{\log(QT)})^2\r).
\end{align}
\end{lemma}

\begin{proof}
Let $t_{0} \in J_{2}$.
Then by formula \eqref{mpro2} for $\frac{1}{2} \leq \s \leq \frac{1}{2} + \frac{3}{2} \a r$, 
we find that
\begin{align*}
\int_{\frac{1}{2} + 14 \a r}^{\s}\frac{L'(x + it_0, \chi)}{L(x + it_0, \chi)}dx 
= \sum_{|\rho - s_0| \leq 20 \a r}\int_{\frac{1}{2} + 14 \a r}^{\s}\frac{dx}{x + it_0 - \rho} + O\l( \a (\log{\log(QT)})^2\r),
\end{align*}
and that 
\begin{align*}
&\log{L(\s + it_0, \chi)} - \log{L\l(s_0, \chi \r)}\\
&=\sum_{|\rho - s_0| \leq 20 \a r}\l( \log(\s + it_0 - \rho) - \log(s_0 - \rho) \r)
+ O\l( \a (\log{\log(QT)})^2\r).
\end{align*}
Hence, by Proposition \ref{main_proposition}, we obtain
\begin{align*}
\log{L(\s + it_0, \chi)}
=&\sum_{|\rho - s_0| \leq 20 \a r}\l( \log(\s + it_0 - \rho) - \log(s_0 - \rho) \r)\\
&+ O\l(\a (\log{\log(QT)})^2\r).
\end{align*}
This completes the proof of Lemma \ref{Rama-Sanka-Lemma5}.
\end{proof}

Now, let us start the proof of Proposition \ref{uniformlowerboundforL}.

\begin{proof}[Proof of Proposition \ref{uniformlowerboundforL}]
If $\frac{1}{2} + 14 \a r \leq \s \leq 2$, then Proposition \ref{uniformlowerboundforL} is implied  by Proposition \ref{main_proposition}.
Hence we consider the case of $\frac{1}{2} \leq \s \leq \frac{1}{2} + 14 \a r$.
We find that
\begin{align*}
\sum_{|\rho - s_0| \leq 20 \a r}\log\l|\frac{\s + it_0 - \rho}{s_0 - \rho}\r|
\geq \sum_{|\rho - s_0| \leq 20 \a r}\log\l|\frac{t_0 - \gamma}{20 \a r}\r|
\geq \sum_{|t_0 - \gamma| \leq 20 \a r}\log\l|\frac{t_0 - \gamma}{20 \a r}\r|
\end{align*}
hold uniformly for $\frac{1}{2} \leq \s \leq \frac{1}{2} + 14 \a r$.
In addition, if $[t_1, t_1 + 1] \subset J_1$, then for $T \geq T_{0}(\a)$, we see that
\begin{align*}
\int_{t_1}^{t_1 + 1}\sum_{|t - \gamma| \leq 20 \a r}\log\l|\frac{t - \gamma}{20 \a r}\r|dt
&=\sum_{t_1 - 20 \a r \leq \gamma \leq t_1 + 1 + 20 \a r}\int_{\max\{ t_1, \gamma - 20 \a r \}}^{\min\{t_1 + 1, \gamma + 20 \a r\}}
\log\l|\frac{t - \gamma}{20 \a r}\r|dt\\
&\geq \sum_{t_1 - 20 \a r \leq \gamma \leq t_1 + 1 + 20 \a r}\int_{\gamma - 20 \a r}^{\gamma + 20 \a r}\log\l|\frac{t - \gamma}{20 \a r}\r|dt\\
&= 20 \a r\sum_{t_1 - 20 \a r \leq \gamma \leq t_1 + 1 + 20 \a r}\int_{-1}^{1}\log|x|dx\\
&\geq  - C \a r \log(QT) \geq - C \a (\log{\log(QT)})^2,
\end{align*}
uniformly for $\chi \in S(Q)$.
Hence, there exists a $t_2 \in [t_1, t_1 + 1]$ satisfying
\begin{align*}
\sum_{|\rho - s_0| \leq 20 \a r}\log\l|\frac{\s + it_2 - \rho}{\frac{1}{2} + \a r + it_2 - \rho}\r|^{-1}
\ll \a (\log{\log(QT)})^2,
\end{align*}
uniformly for $\frac{1}{2} \leq \s \leq 2$ and $\chi \in S(Q)$.
Thus this estimate implies Proposition \ref{uniformlowerboundforL} by formula \eqref{Rama-Sanka-Lemma5-2}.
\end{proof}


\section{\textbf{The finite Euler product appearing in the expression of  
Dirichlet $L$-functions attached to imprimitive characters}}	\label{FEPNC}


In this section, we are going to show some estimates on the function
\begin{align*}
F_{q, \chi^{*}}(s) = \prod_{p \mid q}\l( 1 - \frac{\chi^{*}(p)}{p^{s}} \r)
\end{align*}
including Proposition \ref{unibd_F}.
In the following, we consider the case $F_{q, \chi^{*}} \not\equiv 1$.
In other words, we assume that $\chi^{*}$ is a primitive character modulo $d$ such that 
there exists a prime factor $p$ of $q$ with $p \nmid d$.

\begin{lemma}	\label{F_order}
Let $q \geq 2$.
Then $F_{q, \chi^{*}}$ is an entire function of order $1$.
\end{lemma}

\begin{proof}
By the definition of $F_{q, \chi^{*}}$, if $\s < 0$, then we have
\begin{align*}
\l|F_{q, \chi^{*}}(s)\r| 
&= \prod_{p|q}\l| 1 - \frac{\chi^{*}(p)}{p^{s}} \r|
\leq \prod_{p|q}\l( 1 + \frac{1}{p^{\s}} \r)
= \exp\l(\sum_{p|q}\log\l( 1 + \frac{1}{p^{\s}} \r)\r)\\
&\leq \exp\l( -2\s\sum_{p|q}\log{p} \r)
\leq \exp\l( 2(\log{q})|s| \r).
\end{align*}
On the other hand, if $\s \geq 0$, then we have
\begin{align*}
\l|F_{q, \chi^{*}}(s)\r| 
&= \prod_{p|q}\l| 1 - \frac{\chi^{*}(p)}{p^{s}} \r|
\leq 2^{\omega(q)} \ll_q 1.
\end{align*}
Therefore, order of $F_{q, \chi^{*}}$ is less than or equal to $1$.
In addition, order of  $F_{q, \chi^{*}}$ is greater than or equal to $1$ 
since we can get the following lower bound
\begin{align*}
\l|F_{q, \chi^{*}}(s)\r| 
&= \prod_{p|q}\l| 1 - \frac{\chi^{*}(p)}{p^{s}} \r|
\geq \underset{p \nmid d}{\prod_{p|q}}\l( \frac{1}{p^{\s}} - 1 \r)
=\exp\l(\underset{p \nmid d}{\sum_{p|q}}\log\l(\frac{1}{p^{\s}} - 1 \r)\r)\\
&\geq \exp\l( -\frac{1}{2}\s\sum_{p \nmid q}\log{p} \r)
\geq \exp\l( \frac{1}{2}(\log{2})|\s| \r)
\end{align*}
for $\s \leq -2$.
Hence we obtain Lemma \ref{F_order}.
\end{proof}

The next lemma is an immediate consequence of this lemma.

\begin{lemma}	\label{HFDF}
Let $q \geq 2$ be an integer, $d$ be a proper divisor of $q$, and $\chi^{*}$ be a primitive Dirichlet character modulo $d$.
Then we have
\begin{align}	\label{infprorepF}
F_{q, \chi^{*}}(s) = s^{r}e^{a + bs}\prod_{\eta}\l( 1 - \frac{s}{i\eta} \r)e^{s / i\eta},
\end{align}
where the above infinite product runs over all the zeros of $F_{q, \chi^{*}}$ removing zero at $s = 0$, 
and $r$ is the multiplicity of zero of $F_{q, \chi^{*}}$ at $s = 0$.
In particular, $r$ equals to the number of prime factors of $q$ satisfying $\chi^{*}(p) = 1$. 
Moreover, by taking logarithmic derivative of the both sides, we find that
\begin{align}	\label{nonp_ld_i}
\frac{F_{q, \chi*}'}{F_{q, \chi^{*}}}(s) = \frac{r}{s} + b + \sum_{\eta}\l( \frac{1}{s - i\eta} + \frac{1}{i\eta} \r).
\end{align}
\end{lemma}

\begin{lemma}	\label{b_estimate}
If $b$ is the number appearing in \eqref{infprorepF}, then
\begin{align}	\label{b_formula_2}
b = -\frac{1}{2}\log\l(q' / d'\r) 
+ i\underset{\chi^{*}(p) \not= 1}{\underset{p \nmid d}{\sum_{p|q}}}
\frac{\Im(\chi^{*}(p))}{2 - 2\Re(\chi^{*}(p))}\log{p},
\end{align}
where $q' = \rad(q), d' = \rad(d)$.
\end{lemma}

\begin{proof}
By formula \eqref{nonp_ld_i}, $b$ can be expressed by
\begin{align*}
\lim_{\s \downarrow 0}\l( \frac{F'_{q, \chi^{*}}}{F_{q, \chi^{*}}}(\s) - \frac{r}{\s} \r) = b.
\end{align*}
On the other hand, by taking logarithmic derivatives in \eqref{def_F}, we have
\begin{align*}
\frac{F_{q, \chi*}'}{F_{q, \chi^{*}}}(s)
= \sum_{p|q}\frac{\chi^{*}(p)\log{p}}{p^{s} - \chi^{*}(p)}
= \sum_{i = 1}^{r}\frac{\log{p_{i}}}{p_{i}^{s} - 1} + \underset{\chi^{*}(p) \not= 1}{\sum_{p|q}}\frac{\chi^{*}(p)\log{p}}{p^{s} - \chi^{*}(p)},
\end{align*}
where $p_{i}$ are prime factors of $q$ with $\chi^{*}(p) = 1$.
Therefore, we obtain 
\begin{align*}
\lim_{\s \downarrow 0}\l( \frac{F'_{q, \chi^{*}}}{F_{q, \chi^{*}}}(\s) - \frac{r}{\s} \r)
= -\frac{1}{2}\sum_{i = 1}^{r}\log{p_i} + \underset{\chi^{*}(p) \not= 1}{\sum_{p|q}}\frac{\chi^{*}(p)\log{p}}{1 - \chi^{*}(p)}.
\end{align*}
Moreover, if $\chi^{*}(p) \not= 0, 1$, then we can find that the identity
\begin{align*}
\frac{\chi^{*}(p)}{1 - \chi^{*}(p)}
&= -\frac{1}{2} + i\frac{\Im(\chi^{*}(p))}{2 - 2\Re(\chi^{*}(p))}
\end{align*}
holds by easy calculations.
Thus we obtain Lemma \ref{b_estimate}.
\end{proof}

\begin{lemma}	\label{shortintzerofF}
Let $h > 0$, and $N_{q, \chi^{*}}(t, h)$ be the number of zeros $i\eta$ of $F_{q, \chi^{*}}$ with $t \leq \eta \leq t + h$.
Then we have
\begin{align*}
N_{q, \chi^{*}}(t, h) \leq \omega\l(q' / d'\r) + \frac{h}{2}\log\l(q' / d'\r) + \frac{h^2}{h^2 + t^2}r,
\end{align*}
where $q' = \rad(q)$ and $d' = \rad(d)$.
\end{lemma}

\begin{proof}
By formula \eqref{nonp_ld_i}, we have
\begin{align}	\label{lem_2_1}
\frac{F_{q, \chi*}'}{F_{q, \chi^{*}}}(h + it) = \frac{r}{h + it} + b + \sum_{\eta}\l( \frac{1}{h + it - i\eta} + \frac{1}{i\eta} \r).
\end{align}
On the other hand, by taking logarithmic derivatives in \eqref{def_F}, we find that
\begin{align}	\label{estld_1}
\l|\frac{F_{q, \chi*}'}{F_{q, \chi^{*}}}(h + it)\r|
= \l|\sum_{p|q}\frac{\chi^{*}(p)\log{p}}{p^{h + it} - \chi^{*}(p)}\r|
\leq \underset{p \nmid d}{\sum_{p|q}}\frac{\log{p}}{p^{h} - 1}
\leq \underset{p \nmid d}{\sum_{p|q}}\frac{1}{h} = h^{-1}\omega\l(q' / d'\r).
\end{align}
Now, we take the real parts of the both sides of \eqref{lem_2_1}.
Then, by $\Re\sum_{\eta}(i\eta)^{-1} = 0$ and \eqref{b_formula_2}, we have
\begin{align} \label{zerosumes_2}
\sum_{\eta}\frac{h}{h^2 + (t - \eta)^2}
\leq \frac{\omega\l(q' / d'\r)}{h} + \frac{1}{2}\log\l(q' / d'\r) + \frac{h}{h^2 + t^2}r.
\end{align}
Hence, we have
\begin{align*}
&\frac{\omega\l(q' / d'\r)}{h} + \frac{1}{2}\log\l(q' / d'\r) + \frac{h}{h^2 + t^2}r
\gg \sum_{\eta}\frac{h}{h^2 + (t - \eta)^2}\\
&\geq \sum_{|t - \eta - 1/(2h)| \leq 1/(2h)}\frac{h}{h^2 + (t - \eta)^2}
\gg \sum_{|t - \eta - 1/2| \leq 1/(2h)}\frac{1}{h}
= \frac{1}{h}N_{q, \chi^{*}}(t, h),
\end{align*}
which completes the proof of Lemma \ref{shortintzerofF}.
\end{proof}

\begin{proposition}	\label{trucate_F_1}
Let $h > 0$. If $|\s| \leq h$
, then we have
\begin{align*}
\frac{F_{q, \chi*}'}{F_{q, \chi^{*}}}(s) = \frac{r}{s} + \sum_{|t - \eta| \leq h}\frac{1}{s - i\eta} 
+ O\l(h^{-1}\omega\l(q' / d'\r) + \log\l(q' / d'\r) + \frac{r}{|t| + h} \r).
\end{align*}
\end{proposition}

\begin{proof}
By \eqref{nonp_ld_i} and \eqref{estld_1}, we find that
\begin{align*}
\frac{F_{q, \chi*}'}{F_{q, \chi^{*}}}(\s + it) - \frac{F_{q, \chi*}'}{F_{q, \chi^{*}}}(h + it)
= \frac{r}{\s + it} - \frac{r}{h + it} + \sum_{\eta}\l( \frac{1}{\s + it - i\eta} - \frac{1}{h + it - i\eta} \r),
\end{align*}
and that $\l|\frac{F_{q, \chi*}'}{F_{q, \chi^{*}}}(h + it)\r| \leq h^{-1}\omega\l(q' / d'\r)$.
Therefore, we have
\begin{align*}
\frac{F_{q, \chi*}'}{F_{q, \chi^{*}}}(\s + it)
=&\frac{r}{\s + it} + \l(\sum_{|t - \eta| \leq h} + \sum_{|t - \eta| > h} \r)
\l( \frac{1}{\s + it - i\eta} - \frac{1}{h + it - i\eta} \r)\\
&+O\l( \frac{r}{|t| + h} + h^{-1}\omega\l(q' / d'\r) \r).
\end{align*}
In addition, we can obtain that
\begin{align*}
&\sum_{|t - \eta| > h}\l( \frac{1}{\s + it - i\eta} - \frac{1}{h + it - i\eta} \r)
= \sum_{|t - \eta| > h}\frac{h - \s}{(\s + it - i\eta)(h + it - i\eta)}\\
&\ll \sum_{|t - \eta| > h}\frac{h}{|t - \eta|^2}
\ll \sum_{\eta}\frac{h}{h^2 + (t - \eta)^2}
\leq h^{-1}\omega\l(q' / d'\r) + \frac{1}{2}\log\l(q' / d'\r) + \frac{h}{h^2 + t^2}r
\end{align*}
by \eqref{zerosumes_2}, and that 
\begin{align*}
\sum_{|t - \eta| \leq h}\frac{1}{h + it - i\eta}
\ll \sum_{|t - \eta| \leq h}h^{-1} 
\ll h^{-1}\omega\l(q' / d'\r) + \log\l(q' / d'\r) + \frac{h}{h^{2} + t^2}r
\end{align*}
by Lemma \ref{shortintzerofF}.
Hence we obtain Proposition \ref{trucate_F_1}.
\end{proof}

\begin{proposition}	\label{log_F}
Let $h > 0$.
For $|\s| \leq h, |t| \geq h$, we have
\begin{align}	\label{Eq_log_F}
\log{F_{q, \chi^{*}}}(s)
= \frac{1}{2}\sum_{|t - \eta| \leq h}\log\l( \frac{\s^2 + (t - \eta)^2}
{h^{2} + (t - \eta)^2} \r) + O\l( h^{-1}\omega\l(q' / d'\r) + h\log\l(q' / d'\r) + \frac{h}{|t|}r \r).
\end{align}
\end{proposition}

\begin{proof}
By Lemma \ref{shortintzerofF} and Proposition \ref{trucate_F_1}, we can find that
\begin{align*}
\log{F_{q, \chi^{*}}}(\s + it)
&= \int_{h}^{\s}\frac{F_{q, \chi^{*}}'}{F_{q, \chi^{*}}}(\a + it)d\a + \log{F_{q, \chi^{*}}}(h + it)\\
&= \frac{1}{2}\sum_{|t - \eta| \leq h}\log\l( \frac{\s^2 + (t - \eta)^2}
{h^{2} + (t - \eta)^2} \r) + \log{F_{q, \chi^{*}}(h + it)}\\
& \quad + O\l(\omega(q' / d') + h\log\l(q' / d'\r) + \frac{h}{|t|}r\r).
\end{align*}
In addition, we see that
\begin{align*}
|\log{F_{q, \chi^{*}}}(h + it)|
&=\l|\underset{p \nmid d}{\sum_{p|q}}\sum_{n=1}^{\infty}
\frac{1}{n}\l( \frac{\chi^{*}(p)}{p^{h + it}} \r)^{n}\r|
\leq \underset{p \nmid d}{\sum_{p|q}}\sum_{n = 1}^{\infty}\frac{1}{p^{h n}}
=\underset{p \nmid d}{\sum_{p|q}}\frac{1}{p^h  -1}\\
&\leq \underset{p \nmid d}{\sum_{p|q}}\frac{1}{h \log{p}} \ll h^{-1}\omega\l(q' / d'\r)
\end{align*}
by the definition of $F_{q, \chi^{*}}$ and the Taylor expansion of the logarithmic function.
Hence we obtain Proposition \ref{log_F}.
\end{proof}

Now, let us start the proof of Proposition \ref{unibd_F}.

\begin{proof}[Proof of Proposition \ref{unibd_F}]
Let $T \geq \omega(q)$, and $d$ be the smallest modulus of $\chi \in S_{1}(q)$.
Then, by Lemma \ref{shortintzerofF}, the number of zeros $i\eta$ of the all $F_{q, \chi^{*}}$ with $\chi \in S_{1}(q)$ with $\eta \in [T, T + 1]$
is less than $C\#S_{1}(q)\log\l( q' / d' \r)$, where $C$ is an absolute positive constant.
Therefore, there exists a $t_0 \in [T, T + 1]$ such that $|t_0 - \eta| \geq \frac{1}{2C \#S_1(q)\log\l( q' / d' \r)}$ holds for all zeros $i\eta$.
Now we apply Proposition \ref{log_F} with 
$h \asymp \sqrt{\frac{\omega\l(q' / d'\r) / \log\l(q' / d'\r)}{\log\l( \#S_{1}(q)\omega\l( q' / d'\r) + 2 \r)}}$.
By taking real parts on the both sides of equation \eqref{Eq_log_F}, we obtain
\begin{align*}
&\log|F_{q, \chi^{*}}(\s + it_{0})| 
\geq -\frac{1}{2}\sum_{|t_{0} - \eta| \leq h}\log(4Ch^{2}\#S_1(q)\log\l( q' / d' \r))\\
&\qquad \qquad \qquad \qquad \qquad \qquad
- C'\l(h^{-1}\omega\l(q' / d'\r) + h\log(q' / d')\r)\\
&\geq - C'' \omega\l(q' / d'\r)\log\l(\#S_{1}(q)\omega\l( q' / d' \r) + 2\r)
\l( 1 + \sqrt{\frac{\log\l(q' / d'\r) / \omega\l(q' / d'\r)}
{\log\l( \#S_{1}(q)\omega\l( q' / d'\r) + 2 \r)}} \r),
\end{align*}
uniformly for $\chi \in S_{1}(q)$ and $|\s| \leq h$ with $C', C'' > 0$ sufficiently large positive absolute constants.
Hence, we have
\begin{align*}
&|F_{q, \chi^{*}}(\s + it_{0})|^{-1}\\
&\leq \exp\l(C'' \omega\l(q' / d'\r)\log\l(\#S_{1}(q)\omega\l( q' / d' \r) + 2\r)
\l( 1 + \sqrt{\frac{\log\l(q' / d'\r) / \omega\l(q' / d'\r)}
{\log\l( \#S_{1}(q)\omega\l( q' / d'\r) + 2 \r)}} \r)\r),
\end{align*}
which completes the proof of Proposition \ref{unibd_F}.
\end{proof}


\section{\textbf{Proof of Theorems \ref{exfoML1} and \ref{exfoML2}}}	


\begin{proof}[Proof of Theorem \ref{exfoML1}]
Let $x > 0$, $T \geq \max\l\{T_{0}, \exp\l(q^{1/3}\r), 2 / x\r\}$,
 and $\s_0 = 1 + 1 / \log(x + 3)$.
First, using Perron's formula (cf. Theorem 5.2 and Corollary 5.3 in \cite{MV}), we have
\begin{align}	\label{apply_Perron_L}
M^{*}(x, \chi) 	
=& \frac{1}{2\pi i}\int_{\s_0 -iT_{\nu}}^{\s_0 + iT_{\nu}}\frac{x^{s}}{L(s, \chi)s}ds
+ O\l(\frac{x\log(x + 3)}{T} + \min\l\{1, \frac{x}{T \l< x \r> } \r\} \r),
\end{align}
where $T_{\nu}$ satisfies the inequality
\begin{align}	\label{L_lower}
|L(\s + iT_{\nu}, \chi)|^{-1}
\leq \exp\l(C(\log{\log{T}})^2\r)
\end{align}
for any $\frac{1}{2} \leq \s \leq 2$ and $\chi \in S(q)$ with $T_{\nu} \in [T, 2T]$.
Note that we can take the above $T_{\nu}$ by Proposition \ref{uniformlowerboundforL}.
Here, we remark that $T \geq \exp\l( q^{1/3} \r) \gg q$.

Let $M = m + \frac{1}{2}$ with a positive integer $m$ satisfying $m > T$. By the residue theorem, we have
\begin{align*}	
&\frac{1}{2\pi i}\int_{\s_0 -iT_{\nu}}^{\s_0 + iT_{\nu}}\frac{x^s}{L(s, \chi)s}ds 
= \frac{1}{2\pi i}\l(\int_{-M+iT}^{\s_0 + iT_{\nu}}+\int_{-M-iT_{\nu}}^{-M + iT_{\nu}}+\int_{\s_0 - iT_{\nu}}^{-M - iT_{\nu}}\r)
\frac{x^s}{L(s, \chi)s}ds	\\ 
&	+\sum_{|\gamma|<T_*}\underset{s=\rho}{\Res}\l(\frac{x^s}{L(s, \chi)s}\r)
	+\sum_{0 \leq l < M}\underset{s=-l}{\Res}\l(\frac{x^s}{L(s, \chi)s}\r)\\
&=: J_1 + J_2 + J_3
+\sum_{|\gamma|<T_{\nu}}\underset{s=\rho}{\rm Res}\l(\frac{x^s}{L(s, \chi)s}\r)
+\sum_{0 \leq l < M}\underset{s=-l}{\rm Res}\l(\frac{x^s}{L(s, \chi)s}\r).
\end{align*}
Here, by the basic formula for residues, we find that
\begin{align*}
\underset{s=\rho}{\Res}\l(\frac{x^s}{L(s, \chi)s}\r)
= \frac{1}{(m(\rho)-1)!}\lim_{s \rightarrow \rho}\frac{d^{m(\rho)-1}}{ds^{m(\rho)-1}}
\l((s-\rho)^{m(\rho)}\frac{x^{s}}{L(s, \chi)s}\r).
\end{align*}
As for the other residues, we can also obtain the formula by the functional equation
\begin{align*}
L(s, \chi)
= \epsilon(\chi)L(1 - s, \bar{\chi})2^{s}\pi^{s - 1}d^{1/2 - s}\Gamma(1 - s)\sin\l( \frac{\pi}{2}(s + \kappa) \r).
\end{align*}
Here $\epsilon(\chi)$ is defined by
\begin{align*}
\epsilon(\chi) = \frac{\tau(\chi)}{i^{\kappa}d^{1/2}}.
\end{align*}

Now we estimate the integrals $J_1$, $J_2$ and $J_3$.
By Lemma \ref{RLO}, $J_2$ is evaluated by
\begin{align*}
|J_2|
&=	\l|\int_{|t| \leq T_{\nu}}\frac{x^{-M+it}}{L(-M + it, \chi)(-M + it)}dt \r|
\ll x^{-M}\int_{|t| \leq T_{\nu}}\l( \frac{2 \pi e}{M} \r)^{M}M^{-3/2}dt\\
&\ll \l( \frac{x}{2 \pi e} \r)^{-M}M^{-M - 3 / 2}T.
\end{align*}
Therefore, we have
\begin{align*}
\lim_{M \rightarrow \infty}J_{2} = 0.
\end{align*}

Next we estimate $J_{1}$.
We put
\begin{align*}
J_1
&=\frac{1}{2\pi i}\l(\int_{1/2 + iT_{\nu}}^{\s_0+iT_{\nu}} + \int_{-1 + iT_{\nu}}^{1/2+iT_{\nu}} + \int_{-M+iT_{\nu}}^{-1 + iT_{\nu}}\r)
\frac{x^s}{L(s, \chi)s}ds
=: J_{1}' + J_{1}'' + J_{1}'''.
\end{align*}
By Lemma \ref{RLO} and estimate \eqref{L_lower}, we find that
\begin{align*}
|J_{1}'|
&\ll \int_{1 / 2}^{\s_0} x T_{\nu}^{-1}\exp\l(C(\log{\log{T}})^2\r)d\s 
\ll \frac{x\exp\l( C(\log{\log{T}})^2 \r)}{T},
\end{align*}
\begin{align*}
|{J_1}''|
\ll \frac{\exp\l( C(\log{\log{T}})^2 \r)}{{T}^{3 / 2}}\int_{-1}^{1/2}(xT_\nu)^{\sigma}d\sigma
\ll \frac{x^{1/2}\exp\l( C (\log{\log{T}})^2 \r)}{T\log(xT)},
\end{align*}
and that
\begin{align*}
|{J_1}'''|
\ll \frac{1}{{T}^{3/2}}\int_{-M}^{-1}(xT_{\nu})^{\sigma}d\sigma
\ll \frac{(x T)^{-1}}{{T}^{3 / 2}\log(xT)}.
\end{align*} 
Hence we have
\begin{align*}	
J_1 
\ll \frac{x \exp\l( C(\log{\log{T}})^2 \r)}{T}.
\end{align*}
Similarly, we have
\begin{align*}
J_3 
\ll \frac{x \exp\l( C(\log{\log{T}})^2 \r)}{T}
\end{align*}
since $L(\bar{s}, \chi) = \bar{L(s, \bar{\chi})}$ and $\bar{\chi}$ is also a primitive character. 
From the above estimates, we obtain Theorem \ref{exfoML1}.
\end{proof}


\begin{proof}[Proof of Theorem \ref{exfoML2}]
First, we find that
\begin{align*}
M^{*}(x, \chi) 	
=& \frac{1}{2\pi i}\int_{\s_0 -iT_{\nu}}^{\s_0 + iT_{\nu}}\frac{x^{s}}{L(s, \chi)s}ds
+ O\l(\frac{x\log(x + 3)}{T} + \min\l\{1, \frac{x}{T \l< x \r> } \r\} \r),
\end{align*}
similarly to \eqref{apply_Perron_L}, where we use the same notation as in 
the proof of Theorem \ref{exfoML1}.
Here, by the proof of Proposition \ref{uniformlowerboundforL}, 
$T_{\nu} \in J_{1} \subset J_{0}$ holds, 
where $J_{0}$ and $J_{1}$ are intervals appearing in Corollary \ref{uniesconL} and Lemma \ref{Rama-Sanka-lemma2}, respectively.
We also consider a uniform estimate of $F_{q, \chi^{*}}(s)$ for $\chi \in S^{*}(q)$.
Here $S^{*}(q)$ denotes the set of all imprimitive characters modulo $q$.
If $h := \frac{1}{\log{q}} \leq \s \leq 2$, then we have
\begin{align}	
\prod_{p|q}\l| 1 - \frac{\chi^{*}(p)}{p^{s}} \r|^{-1}
&\leq \prod_{p|q}\l(1 + \frac{1}{p^{h} - 1}\r)
\leq \prod_{p|q}\l(1 + \frac{1}{h\log{p}}\r) \nonumber \\ \label{INEF_Th_2_1}
&
\leq \exp\l(\sum_{p|q} \frac{1}{h\log{p}} \r) 
\leq \exp\l( C\omega(q)\log{q} \r).
\end{align}
Therefore, by this estimate, \eqref{L_lower} and Lemma \ref{RLO}, we obtain
\begin{align}
&|L(\s + iT_{\nu}, \chi)|^{-1} 
= \l|L\l(\s + iT_{\nu}, \chi^{*}\r)\r|^{-1}\prod_{p|q}\l| 1 - \frac{\chi^{*}(p)}{p^{s}} \r|^{-1} \nonumber\\	\label{INEQLF2}
&\leq \exp\l(C\l((\log{\log{T}})^2 + \omega(q)\log{q} \r)\r) \leq \exp\l(C(\log{\log{T}})^2\r)
\end{align}
for $h \leq \s \leq 2$.

Here, let $M = m + \frac{1}{2}$, with a positive integer $m$ satisfying $m > T$.
By Proposition \ref{unibd_F} and $\omega(q) \ll \frac{\log{q}}{\log{\log(q + 5)}}$, 
we can also take some $T_{*} \in [T_{\nu}, T_{\nu} + 1] \subset J_{0}$ such that
\begin{align}	\label{est_Fq}
|F_{q, \chi^{*}}(\s + iT_{*})|^{-1} = \l|\prod_{p|q}\l(1 - \frac{\chi^{*}(p)}{p^{\s + iT_{*}}} \r)\r|^{-1} 
\leq \exp\l(\frac{C(\log{q})^2}{\log{\log(q + 5)}}\r)
\end{align}
holds for $|\s| \leq h$, uniformly $\chi \in S^{*}(q)$.
Then, by using the residue theorem, we have
\begin{align*}	
&\frac{1}{2\pi i}\int_{\s_0 -iT_{\nu}}^{\s_0 + iT_{\nu}}\frac{x^s}{L(s, \chi)s}ds= \frac{1}{2\pi i}\times\\
&\l(\int_{1/4 + iT_{\nu}}^{\s_0 + iT_{\nu}} + \int_{1/4 + iT_{*}}^{1/4 + iT_{\nu}} + \int_{-M + iT_{*}}^{1/4 + iT_{*}}
+ \int_{-M-iT_{*}}^{-M + iT_{*}} + \int_{1/4 - iT_{*}}^{-M - iT_{*}} + \int_{1/4 - iT_{\nu}}^{1/4 - iT_{*}} + \int_{\s_{0} + iT_{\nu}}^{1/4 + iT_{\nu}}\r)\\
&\frac{x^{s}ds}{L(s, \chi)s}
+\sum_{|\gamma|<T_{\nu}}\underset{s=\rho}{\Res}\l(\frac{x^{s}}{L(s, \chi)s}\r)
+\sum_{|\gamma|<T_{*}}\underset{s=i\eta}{\Res}\l(\frac{x^{s}}{L(s, \chi)s}\r)
+\sum_{0 \leq l < M}\underset{s=-l}{\Res}\l(\frac{x^{s}}{L(s, \chi)s}\r).
\end{align*}
Now, we can obtain the residues for non-positive integer in a similar manner as in
Theorem \ref{exfoML1} 
since the all trivial zeros of $L(s, \chi^{*})$ are simple.

We estimate the integrals.
As for the first integral, by inequality \eqref{INEQLF2}, we find that
\begin{align*}
\l|\int_{1 / 4 + iT_{\nu}}^{\s_0 + iT_{\nu}}\frac{x^s}{L(s, \chi)s}ds \r|
\leq \frac{x}{T}\exp\l(C(\log{\log{T}})^2\r).
\end{align*}

Next we consider the fourth integral.
Now, $F_{q, \chi^{*}}(s)$ is estimated by
\begin{align}	\label{INEFm}
|F_{q, \chi^{*}}(\s + it)| = \underset{p \nmid b}{\prod_{p|q}}\l|1 - \frac{\chi^{*}(p)}{p^{\s + it}} \r|
\geq \underset{p \nmid b}{\prod_{p|q}}\l(p^{-\s} - 1\r) 
\geq |\s|^{\omega(q)}\log{2}
\end{align}
for $\s \leq -h$. Here $b$ is the modulo of $\chi^{*}$.
Therefore, by Lemma \ref{RLO} and \eqref{INEFm}, the fourth integral is estimated by
\begin{align*}
\int_{|t| \leq T_{*}}\frac{x^{-M+it}}{L(-M + it, \chi)(-M + it)}dt
&\ll_{q} x^{-M}\int_{|t| \leq T_{*}}\l(\frac{2 \pi e}{M}\r)^{M}M^{-3/2}dt\\
&\ll \l(\frac{x}{2 \pi e}\r)^{-M}M^{-M - 3 / 2}T.
\end{align*}
The last term tends to zero as $M \rightarrow +\infty$.
That is, 
\begin{align*}	
\lim_{M \rightarrow \infty}\int_{|t| \leq T_{*}}\frac{x^{-M+it}}{L(-M + it, \chi)(-M + it)}dt = 0.
\end{align*} 

Next we consider the second integral.
Now, we can see $[T_{\nu}, T_{*}] \subset J_{0}$ since $T_{\nu} \in J_{1}$, and so
we can apply Corollary \ref{uniesconL} to the second integral.
Hence, by Corollary \ref{uniesconL}, \eqref{RLO1} and \eqref{INEF_Th_2_1}, we have
\begin{align*}
\int_{1 / 4 + iT_{*}}^{1/4 + iT_{\nu}}\frac{x^{s}}{L(s, \chi)s}ds
\ll \frac{x^{1/4}}{T}(q^{\omega(q)} \log{T})^{C}
\ll \frac{x^{1/4}}{T}\exp\l(C(\log{\log{T}})^2\r).
\end{align*}

Next we consider the third integral.
We put
\begin{align*}
\int_{-M + iT_{*}}^{1/4 + iT_{*}}\frac{x^{s}}{L(s, \chi)s}ds
= \l(\int_{-1 + iT_{*}}^{1/4 + iT_{*}} + \int_{-M + iT_{*}}^{-1 + iT_{*}}\r)\frac{x^{s}}{L(s, \chi)s}ds.
\end{align*}
Here, we can also apply Corollary \ref{uniesconL} to this case by $T_{*} \in J_{0}$.
By Corollary \ref{uniesconL}, \eqref{RLO1}, \eqref{INEF_Th_2_1}, \eqref{est_Fq} and \eqref{INEFm}, we can find that
\begin{align*}
\int_{-1 + iT_{*}}^{1/4 + iT_{*}}\frac{x^{s}}{L(s, \chi)s}ds
&\ll \frac{\exp\l( C\l(\frac{(\log{q})^2}{\log{\log(q + 5)}} +\log{\log{T}} \r) \r)}{T^{3/2}}\int_{-1}^{1/4}(xT)^{\s}d\s\\
&\ll \frac{x^{1/4}\exp\l(C(\log{\log{T}})^2 \r)}{T^{5 / 4}\log(xT)},
\end{align*}
and that
\begin{align*}
\int_{-M + iT_{*}}^{-1 + iT_{*}}\frac{x^{s}}{L(s, \chi)s}ds
\ll \frac{1}{{T}^{3 / 2}}\int_{-M}^{-1}(xT_{*})^{\sigma}d\sigma
\ll \frac{1}{xT^{5/2}\log(xT)}.
\end{align*}
From the above discussion, the first four integrals are estimated by
\begin{align}	\label{int_es_non_pri}
\ll \frac{x\exp\l( C(\log{\log{T}})^2 \r)}{T}.
\end{align}
The remaining integrals also have the same upper bound since $L(\bar{s}, \chi) = \overline{L(s, \overline{\chi})}$ holds, and 
estimate \eqref{int_es_non_pri} is uniform for $\chi \in S^{*}(q)$.
Thus, we obtain estimate \eqref{APME} for Theorem \ref{exfoML2}.

Finally, we consider the upper bound of the sum of the residues on the imaginary axis.
We separate the sum in three parts such that
\begin{align*}
&\sum_{|\eta| < T_{*}}\Res_{s = i\eta}\l( \frac{x^{s}}{L(s, \chi)s} \r)\\
&= \Res_{s = 0}\l( \frac{x^{s}}{L(s, \chi)s} \r) + \sum_{0 < |\eta| \leq T_{0}(q)}\Res_{s = i\eta}\l( \frac{x^{s}}{L(s, \chi)s} \r)
+ \sum_{T_{0}(q) < |\eta| < T_{*}}\Res_{s = i\eta}\l( \frac{x^{s}}{L(s, \chi)s} \r).
\end{align*}
By estimates \eqref{INEF_Th_2_1}, \eqref{est_Fq}, and \eqref{INEFm}, 
we can find a constant $T_{0}(q)$, depending only on $q$, with
$T_{0}(q) \in [\omega(q) + 5, \omega(q) + 6]$ and
\begin{align*}
|F_{q, \chi^{*}}(\s + iT_{0}(q))|^{-1} \leq \exp\l( \frac{C(\log{q})^2}{\log{\log(q + 5)}} \r)
\end{align*}
for $|\s| \leq 2$. 
By the Leibniz rule, 
 the first sum is estimated by 
\begin{align*}
&\Res_{s = 0}\l( \frac{x^{s}}{L(s, \chi)s} \r)\\
&= \frac{1}{(r + 1 - \kappa)!}\lim_{s \rightarrow 0}\sum_{j = 0}^{r + 1 - \kappa}\begin{pmatrix}
	r + 1 - \kappa \\
	j
\end{pmatrix}
x^{s}(\log{x})^{r + 1 - \kappa - j}\frac{d^{j}}{ds^{j}}\l( \frac{s^{r + 2 - \kappa}}{L(s, \chi)s} \r)\\
&= \frac{(\log{x})^{r + 1 - \kappa}}{L^{(r + 1 - \kappa)}(0, \chi)} + O_{q}\l( (\log{x})^{r - \kappa} \r).
\end{align*}
On the second sum, by the simplicity of zeros of $F_{q, \chi^{*}}$, we have
\begin{align*}
\sum_{0 < |\eta| \leq T_{0}(q)}\Res_{s = i\eta}\l( \frac{x^{s}}{L(s, \chi)s} \r)
= \sum_{0 < |\eta| \leq T_{0}(q)}\frac{x^{i\eta}}{L'(i\eta, \chi)i\eta} = O_{q}(1).
\end{align*}
As for the third sum, we use the result that
\begin{align}	\label{L_upbd}
\frac{1}{L(s, \chi)} \ll \log(q|t|)
\end{align}
for 
\begin{align}	\label{L_zfr}
\s \geq 1 - \frac{c}{\log{q} + (\log(|t|))^{2/3}(\log{\log(|t|)})^{1/3}} \quad \mathrm{and} \quad |t| \geq 5,
\end{align}
where $c$ is a positive absolute constant.
On the region \eqref{L_zfr}, we refer to \cite[\S 9.5]{MTL}.
The author cannot find the above upper bound \eqref{L_upbd} in this region in references, 
but we can obtain it by the standard method  (cf. \cite[\S 11.1]{MV}).
Here, we put $\epsilon(x) = \frac{1}{\log{x}}$ with $x \geq q^{C}\exp\l(C(\log{T})^{2/3}(\log{\log{T}})^{1/3}\r)$.
Then, by the residue theorem, the third sum can be written as
\begin{align*}
&\sum_{T_{0}(q) < \eta < T_{*}}\Res_{s = i\eta}\l( \frac{x^{s}}{L(s, \chi)s} \r)\\
&= \frac{1}{2\pi i}\l( \int_{\epsilon(x) + iT_{0}(q)}^{\epsilon(x) + iT_{*}} + \int_{\epsilon(x) + iT_{*}}^{-\epsilon(x) + iT_{*}}
+ \int_{-\epsilon(x) + iT_{*}}^{-\epsilon(x) + iT_{0}(q)} + \int_{-\epsilon(x) + iT_{0}(q)}^{\epsilon(x) + iT_{0}(q)}  \r)
\frac{x^{s}}{L(s, \chi)s}ds.
\end{align*}
From the definitions of $T_{0}(q)$ and $T_{*}$, we find that
\begin{align*}
\l(\int_{\epsilon(x) + iT_{*}}^{-\epsilon(x) + iT_{*}} + \int_{-\epsilon(x) + iT_{0}(q)}^{\epsilon(x) + iT_{0}(q)}\r)\frac{x^{s}}{L(s, \chi)s}ds
= O_{q}(1).
\end{align*}
On the other hand, by using Lemma \ref{shortintzerofF} and Proposition \ref{log_F} 
with $h' = \sqrt{\frac{\omega(q' / b')}{\log\l(q' / b'\r)\log{\log{x}}}}$, we have
\begin{align*}
|F_{q, \chi^{*}}(\s + it)|
&\geq \exp\l( -C_1\sqrt{\omega\l(q' / b'  \r)\log\l(q' / b'\r)\log{\log{x}}}\r)
\prod_{|t - \eta| \leq h'}|\s + i(t - \eta)|\\
&\geq \exp\l(-C_2\sqrt{\omega\l(q' / b'  \r)\log\l(q' / b'\r)\log{\log{x}}}\r)\l( \frac{1}{\log{x}} \r)^{\omega\l(q' / b'\r)}
\end{align*}
on the lines $|\s| = (\log{x})^{-1}$, $|t| \geq 2$.
Now $b$ denotes the modulus of $\chi^{*}$.
Therefore, for $|\s| = (\log{x})^{-1}, |t| \geq \frac{\log{q}}{\log{\log(q + 2)})} + 5$, we have
\begin{align*}
\frac{1}{L(\s + it, \chi)} 
&\ll \exp\l( C_2\sqrt{\omega\l(q' / b'  \r)\log\l(q' / b'\r)\log{\log{x}}} \r)\frac{(\log{x})^{\omega\l(q' / b' \r)}}{b|t||L(1 + \s + it, \chi^{*})|}\\
&\ll \exp\l( C_2\sqrt{\omega\l(q' / b' \r)\log\l(q' / b'\r)\log{\log{x}}} \r)(\log{x})^{\omega\l(q' / b' \r)}\frac{\log(b|t|)}{b|t|}
\end{align*}
by Lemma \ref{RLO}. Hence we have
\begin{align*}
&\int_{\pm \epsilon(x) + iT_{0}(q)}^{\pm \epsilon(x) + iT_{*}}\frac{x^{s}}{L(s, \chi)s}ds\\
&\ll \exp\l( C_2\sqrt{\omega\l(q' / b'  \r)\log\l(q' / b'\r)\log{\log{x}}} \r)
(\log{x})^{\omega\l(q' / b' \r)}\int_{T_{0}(q)}^{T_{*}}\frac{\log(bt)}{bt^{2}}dt\\
&\ll \exp\l( C_2\sqrt{\omega\l(q' / b'  \r)\log\l(q' / b'\r)\log{\log{x}}} \r)(\log{x})^{\omega\l(q' / b' \r)}.
\end{align*}
Thus we have
\begin{align*}
\sum_{T_{0}(q) < \eta < T_{*}}\Res_{s = i\eta}\l( \frac{x^{s}}{L(s, \chi)s} \r)
\ll \exp\l( C\sqrt{\omega\l(q' / b'  \r)\log\l(q' / b'\r)\log{\log{x}}} \r)(\log{x})^{\omega\l(q' / b'\r)},
\end{align*}
where $C$ is a positive absolute constant.
Similarly, we have
\begin{align*}
\sum_{-T_{*} < \eta < -T_{0}(q)}\Res_{s = i\eta}\l( \frac{x^{s}}{L(s, \chi)s} \r)
\ll \exp\l( C\sqrt{\omega\l(q' / b'  \r)\log\l(q' / b'\r)\log{\log{x}}} \r)(\log{x})^{\omega\l(q' / b' \r)}.
\end{align*}
From the above estimates, we obtain \eqref{IAzeros_bd}.
\end{proof}


\section{\textbf{Proof of Theorem \ref{exfoMDA}}}


First, we prepare some lemmas.

\begin{lemma}	\label{basic_bound}
Let $K$ be any number field, $n_K$ be the degree of $K$, and $\kappa_{K}$ be the residue of $\zeta_K(s)$ at $s = 1$. 
Then, for $\s > 1$, we have
\begin{align}	\label{def_A}
\l|\zeta_K(\s + it)\r| 
\leq \min\l\{\zeta(\s)^{n_K}, \frac{\s}{\s - 1}\kappa_K + \frac{\s \Phi_{0}(K)}{\s - 1 + 1 / n_K}\r\}
=: \Phi_{K}(\s),
\end{align}
where $\Phi_{0}(K)$ is a constant depending only on $K$ such that
\begin{align*}
\l|\sum_{N(\mf{a}) \leq x}1 - \kappa_K x \r| \leq  \Phi_{0}(K)x^{1 - 1 / n_K}.
\end{align*}
\end{lemma}

\begin{proof}
By the Euler product for Dedekind zeta-functions, we find that
\begin{align*}
|\zeta_K(\s + it)| 
&\leq \prod_{\mf{p}}\l( 1 - \frac{1}{N(\mf{p})^{\s}} \r)^{-1}
= \prod_{p}\prod_{\mf{p} \mid p}\l( 1 - \frac{1}{p^{\deg(\mf{p})\s}} \r)^{-1}\\
&\leq \prod_{p}\l( 1 - \frac{1}{p^{\s}} \r)^{-n_K} = \zeta(\s)^{n_K}.
\end{align*}
On the other hand, using the partial summation, we have
\begin{align*}
|\zeta_{K}(\s + it)| 
\leq \s\int_{1}^{\infty}\frac{A(x)}{x^{\s + 1}}dx
\leq \frac{\s}{\s - 1}\kappa_K + \frac{\s \Phi_{0}(K)}{\s - 1 + 1 / n_K},
\end{align*}
where $A(x) = \sum_{N(\mf{a}) \leq x}1$.
Hence we obtain Lemma \ref{basic_bound}.
\end{proof}

\begin{lemma}	\label{RDO}
Let $K$ be a number field. 
If $|t| \geq 1$ and $\s \leq \frac{1}{2}$, then
\begin{align*}	
|\zeta_{K}(s)|
\gg \frac{1}{C^{n_K}} \l( \frac{(2\pi e)^{n_K}}{|d_K||s|^{n_K}} \r)^{\s}\l(|d_K| |s|^{n_K}\r)^{1/2}
\exp\l( n_K|t|\tan^{-1}\l( \frac{1 - \s}{|t|} \r) \r)|\zeta_K(1 - s)|
\end{align*}
If $|t| \leq 1, \s = -(m + 1 / 2)$, then
\begin{align*}	
|\zeta_{K}(s)|
\gg \frac{1}{C^{n_K}} \l( \frac{(2\pi e)^{n_K}}{|d_K||s|^{n_K}} \r)^{\s}\l(|d_K| |\s|^{n_K}\r)^{1/2}|\zeta_K(1 - s)|.
\end{align*}
The above $C$ are absolute constants.
\end{lemma}

\begin{proof}
By the functional equation for $\zeta_K(s)$ which is
\begin{align*}
\zeta_{K}(s) = 2^{n_K}(2\pi)^{n_K(s - 1)}|d_K|^{1/2 - s}(\Gamma(1 - s))^{n_K}
\l( \sin{\frac{\pi s}{2}} \r)^{r_1 + r_2}\l(\cos{\frac{\pi s}{2}}\r)^{r_2}\zeta_{K}(1 - s),
\end{align*}
and the Stirling formula, we obtain this lemma.
\end{proof}

\begin{proof}[Proof of Theorem \ref{exfoMDA}]
Let $x > 0$, $T \geq \max\l\{9, \exp\l(m^{1/3}\r), 2 / x\r\}$ and $\s_0 = 1 + 1 / \log(x + 3)$.
First, using Perron's formula, we have
\begin{align}	\label{apply_Perron_DA}
M_{K}^{*}(x) 	
=& \frac{1}{2\pi i}\int_{\s_0 -iT_{\nu}}^{\s_0 + iT_{\nu}}\frac{x^{s}}{\zeta_K(s)s}ds
+ O\l(\frac{x}{T}\Phi_{K}(\s_0) + \min\l\{1, \frac{x}{T \l< x \r> } \r\} \r),
\end{align}
where $T_{\nu}$ satisfies the inequality
\begin{align}	\label{D_lower}
|\zeta_K(\s + iT_{\nu})|^{-1}
\leq \exp\l(C n_K (\#X(K)) (\log{\log{T}})^2\r)
\end{align}
for any $-1 \leq \s \leq 2$ with $T_{\nu} \in [T, 2T]$ by Corollary \ref{abellowerbound} and Lemma \ref{RDO}.
Here we remark that $T \geq \exp\l( m^{1 / 3} \r) \gg m$.

Let $R = r + \frac{1}{2}$ with positive integer $r$ satisfying $r > T$. By the residue theorem, we have
\begin{align*}	
&\frac{1}{2\pi i}\int_{\s_0 -iT_{\nu}}^{\s_0 + iT_{\nu}}\frac{x^s}{\zeta_{K}(s)s}ds 
= \frac{1}{2\pi i}\l(\int_{-R + iT}^{\s_0 + iT_{\nu}}+\int_{-R - iT_{\nu}}^{-R + iT_{\nu}}+\int_{\s_0 - iT_{\nu}}^{-R - iT_{\nu}}\r)
\frac{x^s}{\zeta_{K}(s)s}ds	\\ 
&	+\sum_{|\gamma|<T_*}\underset{s=\rho}{\Res}\l(\frac{x^s}{\zeta_{K}(s)s}\r)
	+\sum_{0 \leq l < R}\underset{s=-l}{\Res}\l(\frac{x^s}{\zeta_{K}(s)s}\r)\\
&=: J_1 + J_2 + J_3
+\sum_{|\gamma|<T_{\nu}}\underset{s = \rho}{\Res}\l(\frac{x^s}{\zeta_{K}(s)s}\r)
+\sum_{0 \leq l < R}\underset{s = -l}{\Res}\l(\frac{x^s}{\zeta_{K}(s)s}\r).
\end{align*}
Here, by the basic formula for residues, we find that
\begin{align*}
\underset{s=\rho}{\Res}\l(\frac{x^s}{\zeta_{K}(s)s}\r)
= \frac{1}{(m(\rho)-1)!}\lim_{s \rightarrow \rho}\frac{d^{m(\rho)-1}}{ds^{m(\rho)-1}}
\l((s-\rho)^{m(\rho)}\frac{x^{s}}{\zeta_{K}(s)s}\r).
\end{align*}

Next, we estimate the integrals.
By Lemma \ref{RDO}, $J_2$ is evaluated by
\begin{align*}
|J_2|
&=	\l|\int_{|t| \leq T_{\nu}}\frac{x^{-R+it}}{\zeta_K(-R + it)(-R + it)}dt \r|
\ll x^{-R}C^{n_K}\int_{|t| \leq T_{\nu}}\l( \frac{(2 \pi e)^{n_K}}{R^{n_K}} \r)^{R}R^{-n_K / 2 - 1}dt\\
&\ll C^{n_K}\l( \frac{x}{(2 \pi e)^{n_K}} \r)^{-R}R^{-(R - 1 / 2)n_K - 1}T.
\end{align*}
Therefore, we have
\begin{align*}
\lim_{R \rightarrow \infty}J_{2} = 0.
\end{align*}
%
%
Next, we estimate $J_{1}$.
Now, we put
\begin{align*}
J_1
&=\frac{1}{2\pi i}\l(\int_{-1 + iT_{\nu}}^{\s_0+iT_{\nu}} + \int_{-R + iT_{\nu}}^{-1 + iT_{\nu}}\r)
\frac{x^s}{\zeta_K(s)s}ds
=: J_{1}' + J_{1}''.
\end{align*}
By Lemma \ref{basic_bound}, Lemma \ref{RDO} and estimate \eqref{D_lower}, we find that
\begin{align*}
|J_{1}'|
&\ll \int_{-1}^{\s_0} x T_{\nu}^{-1}\exp\l(C n_K(\#X(K))(\log{\log{T}})^2\r)d\s 
\ll \frac{x\exp\l( C n_K (\log{\log{T}})^2 \r)}{T},
\end{align*}
and that
\begin{align*}
|J_{1}''|
\ll \frac{C^{n_K}}{{T}^{3 / 2}}\int_{-R}^{-1}(xT_{\nu})^{\sigma}d\sigma
\ll \frac{C^{n_K} (x T)^{-1}}{{T}^{3/2}\log(xT)}.
\end{align*} 
Hence we have
\begin{align*}	
J_1 
\ll \frac{x \exp\l( C n_K (\#X(K)) (\log{\log{T}})^2 \r)}{T}.
\end{align*}
Similarly, by the Schwarz reflection principle, we have
\begin{align*}
J_3
\ll \frac{x \exp\l( C n_K (\#X(K))(\log{\log{T}})^2 \r)}{T}.
\end{align*}
Thus we obtain estimate \eqref{R'_bound}.

Finally, we consider the upper bound of the sum of the residues for non-positive integers.
We put $r = r_1 + r_2 - 1$.
Then, by the Leibniz rule, we find that
\begin{align*}
&\Res_{s = 0}\l( \frac{x^{s}}{\zeta_K(s)s} \r)
= \frac{1}{r!}\lim_{s \rightarrow 0}\frac{d^{r}}{ds^{r}}\l( \frac{s^{r + 1}}{\zeta_{K}(s)s}x^{s} \r)\\
&= \frac{1}{r!}(\log{x})^{r}\lim_{s \rightarrow 0}\frac{s^{r}}{\zeta_K(s)}
+ \frac{(1 - \delta_{0, r})}{r!}\sum_{j = 1}^{r}\begin{pmatrix}
	r \\
	j
\end{pmatrix}
(\log{x})^{r - j}\lim_{s \rightarrow 0}\frac{d^{j}}{ds^{j}}\l( \frac{s^{r}}{\zeta_K(s)} \r)\\
&= -\frac{(2 \pi)^{n_K} (\log{x})^{r}(2 / \pi)^{r_1 + r_2}}{2^{n_K} |d_K|^{1/2} \kappa_K} 
+ O_{K}\l((1 - \delta_{0, r})|\log{x}|^{r - 1} \r).
\end{align*}
On the other hand, for $l \geq 1$, we have
\begin{align*}
\Res_{s = -l}\l( \frac{x^{s}}{\zeta_K(s)s} \r)
&= \int_{|s + l| = \frac{1}{\log(x + 3)}}\frac{x^{s}}{\zeta_K(s)s}ds\\
&= \frac{i}{\log(x + 3)}\int_{0}^{2 \pi}\frac{x^{-l + \frac{e^{i\theta}}{\log(x + 3)}}e^{i\theta}}
{\zeta_K\l( -l + \frac{e^{i\theta}}{\log(x + 3)} \r)\l( -l + \frac{e^{i\theta}}{\log(x + 3)} \r)}d\theta\\
&\ll \frac{x^{-l}}{l\log(x + 3)}\int_{0}^{2 \pi}\frac{d\theta}{\l|\zeta_K\l( -l + \frac{e^{i\theta}}{\log(x + 3)} \r)\r|}
\end{align*}
by the Cauchy formula.
Now, by the functional equation and the Stirling formula, we can find that if $l$ is even, then
\begin{align*}
\l|\zeta_K\l( -l + \frac{e^{i\theta}}{\log(x + 3)} \r)\r|^{-1}
\ll C^{n_K}\l(\frac{2 \pi e}{l}\r)^{n_K l}l^{-n_K / 2}(\log(x + 3))^{r_1 + r_2},
\end{align*}
and that if $l$ is odd, then
\begin{align*}
\l|\zeta_K\l( -l + \frac{e^{i\theta}}{\log(x + 3)} \r)\r|^{-1}
\ll C^{n_K}\l(\frac{2 \pi e}{l}\r)^{n_K l}l^{-n_K / 2}(\log(x + 3))^{r_2}.
\end{align*}
Hence, for $l \geq 1$, we obtain 
\begin{align*}
\Res_{s = -l}\l( \frac{x^{s}}{\zeta_K(s)s} \r)
\ll \l\{
\begin{array}{ll}
\d{\frac{C^{n_K}x^{-l}}{l^{n_K / 2 + 1}}\l(\frac{2 \pi e}{l}\r)^{n_K l}(\log(x + 3))^{r_1 + r_2 - 1}}	&\text{if \; $l$ is even,}\\
\d{(1 - \delta_{0, r_2})\frac{C^{n_K}x^{-l}}{l^{n_K / 2 + 1}}\l(\frac{2 \pi e}{l}\r)^{n_K l}(\log(x + 3))^{r_2 - 1}}	&\text{if \; $l$ is odd.}
\end{array}
\r.
\end{align*}
Form the above estimates, we obtain Theorem \ref{exfoMDA}.
\end{proof}


\section{\textbf{Proof of Theorem \ref{onemomentforLunconditional}
}}


\begin{proof}[Proof of Theorem \ref{onemomentforLunconditional}]
Let $\chi$ be a primitive Dirichlet character modulo $q$.
Assume the simple zero conjecture for $L(s, \chi)$.
Now, there exisits the domain $-3 / 4 \leq \s \leq 2, 0 < t \leq 2\delta(\chi)$ with $\delta(\chi) \leq 5$
such that this domain does not have zeros of $L(s, \chi)$
since $L(s, \chi)$ is entire function.
In addition, by the compactness of the line segment $-3/4 \leq \s \leq 2, t = \delta$, and 
the continuity of $L(s, \chi)$, we have
\begin{align*}
\frac{1}{|L(s, \chi)|} \leq C(\chi)
\end{align*}
on the same domain, where $C(\chi)$ is a sufficiently large constant depending only on $\chi$.

Let $T \geq \exp\l( q^{1/3} \r)$.
Here, by Proposition \ref{uniformlowerboundforL} and Lemma \ref{RLO}, there exist some $T_{\nu} \in [T, 2T]$ satisfying
\begin{align}	\label{LIN}
|L(\s + iT_{\nu}, \chi)|^{-1}
\leq \exp(C(\log{\log{T}})^2)
\end{align}
for $-1 \leq \s \leq 2$.

Now, by the residue theorem, we have
\begin{align*}
\sum_{0 < \gamma < T_{\nu}}\frac{1}{L'(\rho, \chi)}
 = \frac{1}{2 \pi i}\l(\int_{2 + i\delta(\chi)}^{2 + iT_{\nu}} + \int_{2 + iT_{\nu}}^{-3/4 + iT_{\nu}} 
+ \int_{-3/4 + iT_{\nu}}^{-3/4 + i\delta(\chi)} + \int_{-3/4 + i\delta(\chi)}^{2 + i\delta(\chi)}\r)\frac{ds}{L(s, \chi)}.
\end{align*}
From the way of taking $\delta(\chi)$ and $T_{\nu}$, the integrals on the horizontal line parts are estimated by
\begin{align*}
\l|\int_{-3/4 + iT_{\nu}}^{2 \pm iT_{\nu}}\frac{ds}{L(s, \chi)}\r|
\ll \exp(C (\log{\log{T}})^2),
\end{align*}
and 
\begin{align*}
\l|\int_{-3/4 + i\delta(\chi)}^{2 + i\delta(\chi)}\frac{ds}{L(s, \chi)}\r|
\leq 3C(\chi).
\end{align*}
In addition, we have
\begin{align*}
\l| \int_{-3/4 + i\delta(\chi)}^{-3/4 + iT_{\nu}}\frac{ds}{L(s, \chi)} \r| \ll 1
\end{align*}
since $|L(-3/4 + it, \chi)|^{-1} \ll (|t| + 1)^{-5/4}$ by Lemma \ref{RLO}.
On the first integral term, by the Dirichlet series expression, we have
\begin{align*}
\frac{1}{2 \pi i}\int_{2 + i\delta(\chi)}^{2 + iT_{\nu}}\frac{ds}{L(s, \chi)}
&= \frac{1}{2 \pi}\int_{\delta(\chi)}^{T_{\nu}}\sum_{n = 1}^{\infty}\frac{\mu(n)\chi(n)}{n^{2 + it}}dt\\
&= \frac{T_{\nu}}{2 \pi} + \frac{1}{2 \pi}\sum_{n = 2}^{\infty}\frac{\mu(n)\chi(n)}{n^2}\int_{\delta(\chi)}^{T_{\nu}}n^{-it}dt + O(1)\\
&= \frac{T_{\nu}}{2\pi} + O\l( \sum_{n = 2}^{\infty}\frac{1}{n^2 \log{n}} \r)
=\frac{T_{\nu}}{2\pi} + O(1).
\end{align*}
Hence we have
\begin{align*}
\sum_{0 < \gamma < T_{\nu}}\frac{1}{L'(\rho, \chi)}
= \frac{T_{\nu}}{2 \pi} + O\l(\exp\l(C (\log{\log{T}})^2\r) + C(\chi)\r).
\end{align*}
In particular, for $T \geq T_{0}(q)$ with sufficiently large constant $T_{0}(q)$ depending only on $q$, we obtain
\begin{align*}
\sum_{0 < \gamma \leq 2T}\frac{1}{|L'(\rho, \chi)|}
\geq \sum_{0 < \gamma < T_{\nu}}\frac{1}{|L'(\rho, \chi)|}
\geq \l|\sum_{0 < \gamma < T_{\nu} } \frac{1}{L'(\rho, \chi)} \r| \gg T,
\end{align*}
which completes the proof of Theorem \ref{onemomentforLunconditional}.
\end{proof}

\begin{acknowledgment*}
The author expresses his gratitude to Professor Kohji Matsumoto for his helpful comments.
\end{acknowledgment*}





\end{document}